%% file: main.tex
\pdfoutput=1
\documentclass[11pt]{amsart}

\input{packages}

\usepackage[backend=biber,
	    url=false,
	    doi=false,
	    isbn=false,
	    date=year,
	    maxbibnames=99,
	    giveninits,
	    sortcites=true,
	    sorting=nty,
	    style=numeric]{biblatex}

\renewbibmacro{in:}{}
            
\input{macros}

\input{info}

\input{local}
\title[Effective reproduction number]{The effective reproduction number: 
convexity, concavity and invariance}

\begin{document}

\thanks{This work is partially supported by Labex B\'ezout reference ANR-10-LABX-58}



\begin{abstract} 
  Motivated by the question of  optimal vaccine  allocation strategies in heterogeneous  
  population for   epidemic models, we study  various properties of the  \emph{effective
  reproduction number}. In the simplest case, given a fixed, non-negative matrix $K$, this
  corresponds mathematically to the study of the spectral radius $R_e(\eta)$ of the matrix
  product $\mathrm{Diag}(\eta)K$, as a function of $\eta\in\R_+^n$. The matrix $K$ and the
  vector $\eta$ can be interpreted as a next-generation operator and a vaccination
  strategy.  This can be generalized in an infinite dimensional case where the matrix $K$
  is replaced by a positive integral compact operator, which  is composed  with a
  multiplication by a non-negative function $\eta$.
  
  We  give sufficient conditions  for the function~$R_e$  to be  convex or  a concave.  
  Eventually, we  provide equivalence properties on models  which ensure that the
  function~$R_e$ is unchanged.
\end{abstract}
  
\maketitle

\section{Introduction}

\subsection{The mathematical question}
\label{sec:intro-Q}

For $p\in [1, +\infty ]$, we consider the Lebesgue space $L^p$, with its usual norm   
$\norm{\cdot}_p$, on a $\sigma$-finite measure space $(\Omega, \mathscr{F},  \mu)$. We 
denote by  $\norm{\cdot}_{L^p}$ the  operator norm  on the Banach space  of bounded
operators from $L^p$  to $L^p$. For a bounded operator $T$ on $L^p$, we denote by
$\rho(T)=\lim_{n\rightarrow \infty }\norm{T^n}_{L^p}^{1/n}$ its spectral radius.   We 
recall that  an  operator  $T$  on  $L^p$ is  positive  if $T(L^p_+)\subset L^p_+$,  where
$L^p_+$ denotes the set of non-negative functions in $L^p$.  For $h\in L^\infty _+$, let
$M_h: f \mapsto hf$ denote the bounded operator on $L^p$.

According to the Krein-Rutman theorem, if $T$ is a
positive compact operator on $L^p$ such that $\rho(T)$ is positive, then $\rho(T)$ is also
an eigenvalue.  For such an operator  we define the map $R_e[T]$ on $ L^\infty_+$ by:
\begin{equation}\label{eq:def-ReTh}
  R_e[T](h)=\rho (T \, M_h).
\end{equation}
By homogeneity of the spectral radius, the study of the map $R_e[T]$, it is enough to consider
this map only on the subset $\Delta\subset L^\infty _+$ of non-negative measurable  
functions bounded by 1.  Our aim is to provide sufficient conditions on $T$ for the map
$R_e[T]$ to be convex or concave on $\Delta$. We briefly explain in the next section how
this question is related to the optimal vaccination problem in epidemic models. 

\subsection{The epidemic motivation}

In finite metapopulation  models,  the  population is  divided  into
$N  \geq 2$  different sub-populations; this  amounts to considering the  discrete state
space  $\Omega_{\mathrm{d}} =  \{ 1,  \ldots,  N \}$. 
Following \cite{hill-longini-2003}, the entry $K_{ij}$ of the so-called next-generation
matrix $K$ is equal to the expected number of secondary infections for people in subgroup
$i$ resulting from a single randomly selected non-vaccinated infectious person in subgroup
$j$. The matrix $K$ has non-negative entries, and represents the compact positive
operator $T$.  Let  $\eta\in \Delta=[0, 1]^N$ represent a vaccination strategy, that is, $\eta_i$ is the
fraction of non-vaccinated individuals in the $i$\textsuperscript{th}
sub-population; thus $\eta_i = 0$ when the $i$\textsuperscript{th} sub-population
is fully vaccinated, and $1$ when it is not vaccinated at all --- this seemingly unnatural
convention is in particular motivated by the simple form of
Equation~\eqref{eq:def-ReTh}. So, the strategy $\un\in \Delta$, with all its entries
equal to 1, corresponds to an entirely non-vaccinated population.

The \emph{effective reproduction number}~$R_e[K](\eta)$ associated to the vaccination
strategy~$\eta$ is  then the spectral radius  of the matrix    $ K\cdot  \Diag(\eta)$:
\begin{equation}
  \label{eq:informalPb}
  R_e[K](\eta)=\rho(  K\cdot  \Diag(\eta)),
\end{equation}
where $\Diag(\eta)$ is the diagonal matrix with diagonal entries $\eta$. It may  be
interpreted as the mean number of infections coming  from a typical case in the  SIS model
(where ``S'' and ``I'' stand for susceptible  and infected). In particular, we denote by 
$R_0=R_e[K](\un)$  the  so-called \emph{basic  reproduction  number} associated to the
metapopulation  epidemiological model, see Lajmanovich and 
Yorke~\cite{lajmanovich1976deterministic}. Let  us mention  that in this  model  if 
$R_0\leq  1$,  then there  is  no  endemic  equilibrium (\emph{i.e.}, the epidemic
vanishes asymptotically), whereas if $R_0>1$, there exists at  least one non-trivial
endemic  equilibrium (which means that  the  epidemic is  persistent).   With  the
interpretation  of  the function~$R_e$ in mind,  it is then very natural to  minimize it
under a constraint on the cost of the vaccination strategies $\eta$.
This constrained optimization problem appears  in most of the literature for  designing
efficient  vaccination strategies  for multiple  epidemic situation (SIS/SIR/SEIR)
\cite{EpidemicsInHeCairns1989, hill-longini-2003,  TheMostEfficiDuijze2016,
  CriticalImmuneMatraj2012, poghotanyan_constrained_2018,               
OptimalInfluenEnayat2020, IdentifyingOptZhao2019,  ddz-theo}.   Note  that  in  some   of 
these references, the effective reproduction number is defined as the spectral radius of 
the matrix  $\Diag(\eta) \cdot K$. Since the  eigenvalues of $\Diag(\eta)  \cdot  K$  are
exactly  the  eigenvalues  of  the  matrix $K\cdot  \Diag(\eta)$, this  actually defines 
the same  function $R_e[K]$.

Given the importance of convexity to solve optimization problems efficiently, it is
natural to look for conditions on the matrix $K$ that imply convexity or concavity for the
map~$R_e[K]$ defined by \eqref{eq:informalPb}. Those properties can be useful to design
vaccination strategies  in the best possible way; see the companion papers~\cite{ddz-reg,
ddz-mono}.

\subsection{The finite dimensional case}

 In their investigation of the behavior of the 
map $R_e[K]$ defined in~\eqref{eq:informalPb}, Hill and Longini conjectured
in~\cite{hill-longini-2003} sufficient spectral conditions to get either concavity or
convexity. More precisely, guided by explicit examples, they state that $R_e[K]$ should be
convex if all the eigenvalues of $K$ are non negative real numbers, and that it should be
concave if all eigenvalues are real, with only one positive eigenvalue.

Our first series of results show that, while this conjecture cannot hold in full
generality -- see Section~\ref{sec:comparison_Hill_Longini} -- it is true under an
additional symmetry hypothesis. Recall that a matrix $K$ is called diagonally
symmetrizable if there exist positive numbers $(d_1,\ldots d_{N})$ such that for all
$i,j$, $d_i K_{ij} = d_j K_{ji}$. Such a matrix is diagonalizable with real eigenvalues
according to the spectral theorem for symmetric matrices. The following result, which
appears below in the text as Theorem~\ref{th:hill-longini-meta}, settles the conjecture
for diagonally symmetrizable matrices. Let us mention that the eigenvalue $\lambda_1$ in
the theorem below is non-negative and is equal to the spectral radius of $K$, that is,
$\lambda_1=R_e[K](\un)=R_0$, thanks to the Perron-Frobenius theory. We consider the
function $R_e=R_e[K]$ defined on $[0, 1]^N$. 

\begin{theorem}
  Let $K$ be an $N\times N$ matrix with non-negative entries. Suppose that $K$ is
  diagonally symmetrizable with eigenvalues $\lambda_1\geq\lambda_2\cdots \geq
  \lambda_{N}$.
  \begin{enumerate}[(i)]
  \item\label{item:cairns} If $\lambda_{N}\geq 0$, then the function $R_e$ is convex. 
  \item If $\lambda_2 \leq 0$, then the function $R_e$ is concave. 
  \end{enumerate}
\end{theorem}

Note that the case~\ref{item:cairns} appears already in
Cairns~\cite{EpidemicsInHeCairns1989}; see 
Section~\ref{sec:comparison_Hill_Longini} below for a detailed comparison with existing
results. This completes results on log-convexity of the map $R_e[K]$
given in  \cite{Fried80,feng_elaboration_2015}. 
Notice also that if $K$ and $K'$ are diagonally similar up to transposition, they
define the same function~$R_e$; see \cite{ddz-pres} for more results in this direction.
Eventually, the concavity of the map $R_e[K]$ implies that $K$ has a
unique irreducible component in its Perron-Frobenius diagonalization
as shown in Lemma~\ref{lem:conc-mono} below.

\subsection{The general case}

We    now    give    our    main    result    in    the    setting    of
Section~\ref{sec:intro-Q}. We  give in  Definition~\ref{def:diag-sym} an extension  to the
notion  of ``diagonally  symmetrizable'' for  compact operators. For example, according to
Proposition~\ref{prop:sylvester}, if $T'$ is a self-adjoint compact operator on  $L^2$
and~$f,g$ are two non-negative measurable functions  defined on~$\Omega$ bounded and 
bounded away from $0$,  then the  operator  $T=M_f  \, T'  M_g$  is  a compact  diagonally
symmetrizable on $L^2$. In particular, Corollary~\ref{cor:diag-sym=spec-reel} states that
diagonally symmetrizable compact operators  on $L^p$, with $p\in  [1, +\infty )$, have a
real spectrum . 
   
For a  compact operator $T$,  let $\pos(T)$ (resp.\ $\nega(T)$) denote the number of 
eigenvalues with positive (resp.\ negative) real part  taking into account their 
(algebraic) multiplicity.  Then, we obtain the following result given in
Theorem~\ref{th:hill-longini} below.

\begin{theorem} [Convexity/Concavity of $R_e$]
  Let $T$  be a  positive compact  diagonally symmetrizable  operator on
  $L^p$ with $p\in [1, +\infty )$. We consider the function $R_e=R_e[T]$
  defined on~$\Delta$.
\begin{enumerate}[(i)]
   \item\label{item:hill-longini-cvx-intro} If $\nega(T)=0$, then the function $R_e$
     is convex.
   \item\label{item:hill-longini-cave-intro} If
     $\pos(T)=1$, 
     then the function $R_e$ is concave.
   \end{enumerate}
 \end{theorem}
 The proof of  the concavity property relies on  the explicit expression
 of the second  derivative of $R_e[T]$ when $T$ is  self-adjoint and the
 uses of the Sylvester's inertia theorem.

The  concavity  property of  $R_e[T]$  implies  a strong  structural property on the
operator $T$. In  order to establish this result, we present  in
Section~\ref{sec:reducible}  an atomic  decomposition of the    space    $\Omega$   
related    to    the    operator    $T$ following~\cite{schwartz61}.
In  particular,  we extend  the  notion  of
quasi-irreducible operator to  the non self-adjoint case  and say an
operator is  monatomic if  it has  only one  non-trivial irreducible
component; see       Definition~\ref{def:monatomic}        in
Section~\ref{sec:q-irr-monat}.   If   $T$  is  a   positive  compact
operator   on   $L^p$  for   some   $p\in   [1,  +\infty   )$   with
$R_0=R_e[T](\un)>0$, where $\un \in \Delta$ is the constant function
equal to 1, then we have the following properties:
\begin{enumerate}[(i)]
  \item If $R_e[T]$ is concave, then $T$ is monatomic according to Lemma~\ref{lem:conc-mono}.
  \item If  $\pos(T)=1$, then  $R_0$ is
    simple and the only eigenvalue in $\R_+^*$, and thus $T$ is monatomic according to
    Lemma~\ref{lem:R0simple}.
  \item More generally, using the  decomposition of a reducible operator
    from      Lemma     \ref{lem:Rei},      we      get     that      if
    $\spec(T) \subset \R_-\cup\{R_0\}$ and  $T$ is a  diagonally
    symmetrizable operator,  then the function  $R_e$ is the  maximum of
    $m$  concave  functions  which are  non-zero  on  $m$
    pairwise disjoint subsets of $\Delta$, where $m$ is the (algebraic)
    multiplicity of $R_0$.
\end{enumerate}

Eventually, by considering a general positive compact operator $T$  on
$L^p$  for   some   $p\in   [1,  +\infty   )$, following
~\cite{schwartz61}, we provide in Corollary~\ref{cor:Speci} the decomposition $R_e[T]$  on the
irreducible atoms: 
   \[
      R_e[T]= \max_{i\in I}   R_e[T_i] ,
    \]
where $  T_i(\cdot)=\ind{\Omega_i} \, T\, (\ind{\Omega_i} \cdot)$ with
$(\Omega_i, i\in I)$ the at most countable collection of irreducible atoms in
$\Omega$ associated to $T$. 

\subsection{Structure of the paper}

After      recalling       the      mathematical       framework      in
Section~\ref{sec:settings}, we discuss invariance properties of $R_e$ in
Section~\ref{sec:equivalent}. The convexity properties  of $R_e$ and the
related   conjecture   of   Hill    and   Longini   are   discussed   in
Section~\ref{sec:Hill_Longini}.  Finally, the  case of reducible operators
is   treated  in   Section~\ref{sec:reducible},   using  the   Frobenius
decomposition from \cite{schwartz61}.

\section{Setting, notations and previous results}\label{sec:settings}

\subsection{Spaces, operators, spectra}\label{sec:spaces}

All metric spaces~$(S,d)$ are endowed with their Borel~$\sigma$-field denoted by $\cb(S)$.
The set $\ck$ of compact subsets of~$\C$ endowed with the Hausdorff distance
$d_\mathrm{H}$ is a metric space, and the function~$\mathrm{rad}$ from~$\ck$ to $\R_+$
defined by~$\mathrm{rad}(K)=\max\{|\lambda|\, ,\, \lambda\in K\}$ is Lipschitz continuous
from~$(\ck,d_\mathrm{H})$ to~$\R$ endowed with its usual Euclidean distance.

Let~$(\Omega,  \cf,  \mu)$  be  a measured space, with $\mu$ a $\sigma$-finite
(positive and non-zero) measure.   For~$f$  and~$g$
real-valued      functions     defined      on~$\Omega$,     we      may
write~$\langle f, g \rangle$ or $\int_\Omega  f g \, \mathrm{d} \mu$ for
$\int_\Omega  f(x) g(x)  \,\mu( \mathrm{d}  x)$ whenever  the latter  is
meaningful.     For~$p    \in    [1,    +\infty]$,    we    denote    by
$L^p=L^p(  \mu)=L^p(\Omega, \mu)$  the space  of real-valued  measurable
functions~$g$            defined on~$\Omega$           such            that
$\norm{g}_p=\left(\int |g|^p  \, \mathrm{d} \mu\right)^{1/p}$  (with the
convention  that~$\norm{g}_\infty$  is the~$\mu$-essential  supremum  of
$|g|$)   is  finite,   where  functions   which  agree~$\mu$-a.e.\   are
identified.  We denote  by~$L^p_+$ the  subset of~$L^p$  of non-negative
functions.   We  define~$\Delta$  as the   subset  of  $L^\infty  $  of
$[0, 1]$-valued measurable functions defined on~$\Omega$.
We denote by $\un$ (resp.\ $\zero$) the constant function on $\Omega$
equal to $1$ (resp.\ $0$).

Let~$(E,   \norm{\cdot})$  be   a  complex   Banach  space.   We  denote
by~$\norm{\cdot}_E$ the operator norm on~$\cll(E)$ the Banach algebra of
bounded operators. The spectrum~$\spec(T)$  of~$T\in \cll(E)$ is the set
of~$\lambda\in \C$ such that~$ T -  \lambda \mathrm{Id}$ does not have a
bounded inverse  operator, where~$\mathrm{Id}$ is the  identity operator
on~$E$. Recall that~$\spec(T)$ is a compact subset of~$\C$, and that the
spectral radius of~$T$ is given by:
\begin{equation}\label{eq:def-rho}
  \rho(T)=\mathrm{rad}(\spec(T))=
  \lim_{n\rightarrow \infty } \norm{T^n}_E^{1/n}.
\end{equation}
The element $\lambda\in \spec(T)$ is an eigenvalue if there exists $x\in E$ such that
$Tx=\lambda x$ and $x\neq 0$. Following \cite{konig}, we define the (algebraic)
multiplicity of $\lambda\in \C$ by:
\[
  \mult(\lambda, T)= \dim \left( \bigcup_{k\in \N^*} \ker (T- \lambda
  \mathrm{Id})^k\right),
\]
so that $\lambda$ is an eigenvalue if $\mult(\lambda, T)\geq 1$. We say the eigenvalue
$\lambda$ of $T$ is \emph{simple} if $\mult(\lambda, T)=1$. 

If~$E$ is also an algebra of functions,  for~$g \in E$, we denote by~$M_g$
the multiplication operator  (possibly unbounded) defined by~$M_g(h)=gh$
for all~$h \in E$; if furthermore $g$ is the  indicator function of a
set $A$, we simply
write $M_A$ for $M_{\ind{A}}$.

\subsection{Invariance and continuity of the spectrum for compact operators}\label{sec:inv-cont}

We  collect some  known  results  on the  spectrum  and multiplicity  of
eigenvalues related  to compact operators. Let~$(E,  \norm{\cdot})$ be a
complex  Banach space.  Let~$A\in \cll(E)$.  We denote  by $A^\top$  the
adjoint of  $A$. A  sequence~$(A_n,n \in  \N)$ of  elements of~$\cll(E)$
converges          strongly          to~$A         \in          \cll(E)$
if~$\lim_{n\rightarrow  \infty   }  \norm{A_nx  -Ax}=0$   for  all~$x\in
E$. Following~\cite{anselone}, a set  of operators~$\ca \subset \cll(E)$
is          \emph{collectively          compact}         if          the
set~$\{ A x \, \colon \, A  \in \ca, \, \norm{x}\leq 1 \}$ is relatively
compact. Recall  that the spectrum  of a  compact operator is  finite or
countable   and  has   at  most   one  accumulation   point,  which   is
$0$. Furthermore,  $0$ belongs to  the spectrum of compact  operators in
infinite dimension.

We refer  to \cite{schaefer_banach_1974}  for an introduction  to Banach
lattices and positive operators; we  shall only consider the real Banach
lattices $L^p=L^p(\Omega, \mu)$  for $p\in [1, +\infty ]$  on a measured
space $(\Omega, \cf,  \mu)$ with a $\sigma$-finite  positive non-zero measure, as
well as their complex extension. (Recall that the norm of an operator on
$L^p$    or   its    natural    complex   extension    is   the    same,
see~\cite[Corollary~1.3]{complex}).
A  bounded operator $A$ on $L^p$  is positive
if  $A(L^p_+)\subset  L^p_+$. 

We  say  that   two  complex  Banach  spaces   $(E,  \norm{\cdot})$  and $(E',        
\norm{\cdot}')$          are         compatible         if $(E\cap  E',  \norm{\cdot}+ 
\norm{\cdot}')$  is  a  Banach  space,  and
$E   \cap  E'$   is   dense  in   $E$  and   in   $E'$. Given two compatible spaces $E$
and $E'$,   two   operators  $A\in  
\cll(E)$  and
$A'\in   \cll(E')$ are said to be   consistent   if,   with   
$E''=E\cap   E'$, $A(E'')\subset  E''$,   $A'(E'')\subset  E''$   and  $Ax=A'x$   for  all
$x\in E''$.

\begin{lemma}[Spectral properties]
  \label{lem:prop-spec-mult}
  Let $A, B$ be elements of $\cll(E)$. 
  \begin{enumerate}[(i)]
  \item\label{item:A-B} If $E$ is a Banach lattice, and if~$A$,~$B$ and~$A-B$ are positive
    operators, then we have:
    \begin{equation}\label{eq:r(A)r(B)} \rho(A)\geq \rho(B). 
    \end{equation}
  \item\label{item:spec-adjoint-mult} If $A$ is compact, then $A^\top$,  $AB$ and $BA$
    are compact and we have:
    \begin{align}\label{eq:adjoint-mult}
      \spec(A)=\spec(A^\top) &\quad\text{and}\quad
      \mult(\lambda,A)=\mult(\lambda, A^\top) \quad\text{for $\lambda\in
      \C^*$},\\
      \label{eq:r(AB)-mult}
      \spec(AB)=\spec(BA) & \quad\text{and}\quad
      \mult(\lambda, AB)=\mult(\lambda, BA)\quad\text{for $\lambda\in \C^*$},
    \end{align}
    and in particular:
    \begin{equation}\label{eq:r(AB)}
      \rho(AB)=\rho(BA). 
    \end{equation}
  \item \label{item:density-mult} Let~$(E', \norm{\cdot}')$ be a complex
    Banach space and $A'\in \cll(E')$  such
    $(E, \norm{\cdot})$ and $(E',
\norm{\cdot}')$ are compatible, and $A$ and $A'$ are consistent.  If $A$
    and $A'$ are compact, then we have:
    \begin{equation}\label{eq:sT=sT'}
      \spec(A)=\spec(A')
      \quad\text{and}\quad
      \, \mult(\lambda, A)=\mult(\lambda, A')\quad\text{for $\lambda\in
      \C^*$}.
    \end{equation}
  \item \label{item:collectK-cv} Let~$(A_n, n\in \N)$ be a collectively compact sequence
    which converges strongly to~$A$. Then, we have $\lim_{n\rightarrow \infty }
    \spec(A_n)=\spec(A)$ in~$(\ck, d_\mathrm{H})$, $\lim_{n\rightarrow }
    \rho(A_n)=\rho(A)$ and for $\lambda\in
    \spec(A)\cap \C^*$, $r>0$ such that $\lambda'\in \spec(A)$ and
    $|\lambda-\lambda'|\leq r$ implies $\lambda=\lambda'$, and
    all $n$ large enough:
    \begin{equation}
      \label{eq:collectK-mult}
      \mult(\lambda, A)= \sum_{\lambda'\in
      \spec(A_n),\, |\lambda-\lambda'|\leq r} \mult(\lambda', A_n).
    \end{equation}
  \end{enumerate}
\end{lemma}

\begin{proof}
  Property       \ref{item:A-B}       can        be       found       in
  \cite[Theorem~4.2]{marek}. Property~\ref{item:density-mult} is
  in~\cite[Theorem~4.2.15]{Davies07}.

  \medskip
       
  Equation   \eqref{eq:adjoint-mult}   from
  Property  \ref{item:spec-adjoint-mult}   can  be  deduced   from  from
  \cite[Theorem~p.\ 20]{konig}. Using \cite[Proposition~p.\ 25]{konig}, we
  get     the    second     part     of    \eqref{eq:r(AB)-mult}     and
  $ \spec(AB) \cap \C^*=\spec(BA)  \cap \C^*$, and thus \eqref{eq:r(AB)}
  holds.  To  get the  first  part  of \eqref{eq:r(AB)-mult},  we only
  need to consider if $0$ belongs to the spectrum or not. We first 
  consider  the infinite  dimensional case:  as~$A$ is  compact, we  get
  that~$AB$ and~$BA$  are compact,  thus~$0$ belongs to  their spectrum.
  We     then     consider      the     finite     dimensional     case:
  as~$\mathrm{det}(AB)=\mathrm{det}(A)\mathrm{det}(B)
  =\mathrm{det}(BA)$, where~$A$  and~$B$ denote  also the matrix  of the
  corresponding operator in a given base, we get that~$0$ belongs to the
  spectrum of~$AB$ if and only if it belongs to the spectrum of~$BA$.

  \medskip

  We eventually  check Property  \ref{item:collectK-cv}. We  deduce from
  \cite[Theorems~4.8 and  4.16]{anselone} (see also (d),  (g) [take care
  that  $d(\lambda,  K)$  therein   is  the  algebraic  multiplicity  of
  $\lambda$  for  the  compact  operator   $K$  and  not  the  geometric
  multiplicity]  and (e)  in~\cite[Section~3]{SpectralProperAnselo1974})
  that~$\lim_{n\rightarrow    \infty    }    \spec(A_n)=\spec(A)$    and
  \eqref{eq:collectK-mult}.  Then  use that  the function~$\mathrm{rad}$
  is continuous  to deduce the  convergence of the spectral  radius from
  the convergence of the spectra.
\end{proof}

We complete  this section with  an example of compatible  Banach spaces.
According to \cite[Problem~2.2.9~p.\ 49]{Davies07},  the spaces $L^p(\mu)$
are compatible for all $p\in [1,  +\infty )$. We shall use the following
slightly  more  general  result.   We recall  that  two  $\sigma$-finite
measures  on  $(\Omega,  \cf)$,  say  $\mu $  and  $\nu$,  are  mutually
absolutely    continuous     if    for     $A\in    \cf$,     we    have
$\mu(A)=0  \Longleftrightarrow \nu(A)=0$.  Thanks  to the  Radon-Nikodym
theorem,  the $\sigma$-finite  measures  $\mu $  and  $\nu$ are  mutually
absolutely  continuous if  and only  if there  exists a  positive finite
measurable function $h$ such that $\rd \nu = h \, \rd \mu$.

\begin{lemma}[Compatibility of $L^p$ spaces]\label{lem:Lp-compatible}
  Let  $\mu $  and $\nu$ be two $\sigma$-finite measures on $(\Omega, \cf)$ which are
  mutually absolutely   continuous, and let $p,r\in [1, +\infty )$. Then, the spaces
  $L^p(\mu)$ and $L^r(\nu)$ are compatible. 
\end{lemma}

\begin{proof}
  First note that a property is true $\mu$-a.e.\ if and only if it is true $\nu$-a.e.\
  since $\mu$ and $\nu$ are mutually absolutely continuous. Hence, we shall simply write
  that that the property is true a.e.\ in this case.

  Let us prove that $L^p(\mu) \cap L^r(\nu)$ is dense in $L^p(\mu)$. Let $f\in
  L^r(\nu)$ such that $f>0$ a.e. For any $g\in L^p_+(\mu)$ note that the  non-decreasing
  sequence $(\min (g, nf), n\in \N)$ of  elements of  $L^p(\mu)\cap L^r(\nu)$  converges
  towards $g$ a.e.; and so, it converges in $L^p(\mu)$ according to the dominated
  convergence theorem. This gives $ L^p(\mu)\cap L^r(\nu)$ is dense in $L^p(\mu)$ and in
  $L^r(\nu)$ by symmetry.

  To prove that $L^p(\mu) \cap L^r(\nu)$ is complete (with respect to the norm given by the sum of
  the norms in $L^p(\mu)$ and $L^r(\nu)$), it is enough to check that if a sequence $(h_n,
  n\in \N)$ converges to $g$ in $L^p(\mu)$ and to $f$ in $L^r(\nu)$, then $g=f$ a.e.  This
  is immediate: for such a sequence, one can extract a sub-sequence which converges to $g$
  a.e.\ and to $f$ a.e.
\end{proof}

\subsection{The effective reproduction number~%
  \texorpdfstring{$R_e$}{Re}}\label{sec:weak-topo}

For  $p\in  [1,  +\infty  )$  and~$\eta\in  \Delta$  the  multiplication
operator~$M_\eta$ is bounded, and if $T$  is a compact operator on $L^p$
then so is  $ T M_\eta$.  Following~\cite{ddz-theo}  where only integral
operators where considered, and keeping similar notations, we define the
\emph{reproduction number}  associated to the positive  compact operator
$T$ (on $L^p$ for some $p\in [1, +\infty )$) as its spectral radius:
\begin{equation}\label{eq:def-R0}
  R_0[T]=\rho(T),
\end{equation}
the \emph{effective spectrum} function~$\spec[T]$ from~$\Delta$ to~$\ck$ by:
\begin{equation}\label{eq:def-sigma_e}
  \spec[T](\eta)=\spec(T M_\eta), 
\end{equation}
and the \emph{effective reproduction number} function $R_e[T]=\mathrm{rad}\circ \spec[T]$
from~$\Delta$ to~$\R_+$ by:
\begin{equation}\label{eq:def-R_e}
  R_e[T](\eta)=\mathrm{rad}( \spec(TM_\eta))=\rho(TM_\eta). 
\end{equation}
Take care that:
\[
  \spec(T)=\spec[T](\un)\quad\text{and}\quad R_0[T]=R_e[T](\un).
\]
When there is no risk of confusion on the positive compact operator $T$, we simply write $R_e$ and $R_0$
for the function~$R_e[T]$ and the number~$R_0[T]$.
We have the following immediate properties for the function $R_e[T]$
(use Lemma~\ref{lem:prop-spec-mult}~\ref{item:A-B}
for the third property).

\begin{proposition}[Elementary properties of $R_e$]\label{prop:R_e}
The function~$R_e=R_e[T]$, where $T$ is a  positive compact operator on
$L^p$ with $p\in [1, +\infty )$ satisfies the
following properties:
  \begin{enumerate}[(i)]
  \item\label{prop:a.e.-Re}%
    $R_e(\eta_1)=R_e(\eta_2)$ if~$\eta_1=\eta_2,\, \mu\as$, and~$\eta_1, \eta_2\in
    \Delta$,
  \item\label{prop:min_Re}%
    $R_e(\zero) = 0$ and~$R_e(\un) = R_0$,
  \item\label{prop:increase}%
    $R_e(\eta_1) \leq R_e(\eta_2)$ for all~$\eta_1, \eta_2\in \Delta$ such
    that~$\eta_1\leq \eta_2$,
  \item\label{prop:normal}%
    $R_e(\lambda \eta) = \lambda R_e(\eta)$, for all~$\eta \in \Delta$ and~$\lambda \in
    [0,1]$.
  \end{enumerate}
\end{proposition}

We shall use the following continuity property of the spectrum; see also
\cite[Proposition~3.6]{ddz-theo} for stronger results when considering integral operators
and the weak topology on $\Delta$.

\begin{lemma}[Continuity of the spectrum]
  \label{lem:cont-spec}
  Let   $T$   be   a   compact   operator   on   $L^p$   with
  $p\in [1, +\infty )$. Let $(v_n, n\in \N)$ and $(w_n, n\in \N)$ be two
  bounded sequences in $L ^\infty $ which converge respectively to $v_\infty $ and $w_\infty $,
  and let $T_n=M_{v_n}\, T \, M_{w_n}$.
  Then  for any $\eta\in\Delta$, as $n$ goes to infinity, we have that:
  \begin{enumerate}[(i)]
  \item $\spec[T_n](\eta) $ converges to $\spec[T_\infty ](\eta) $ in $\ck$,
  \item $R_e[T_n](\eta)$ converges to $R_e[T_\infty ](\eta)$ in $\R$,
  \item for   any   $\lambda\in\spec(T_\infty M_\eta)\cap \C^*$ and any $r>0$
    such that $\lambda'\in \spec(T_\infty M_\eta)$ and
    $|\lambda-\lambda'|\leq r$ implies $\lambda=\lambda'$, then for
    all $n$ large enough:
    \begin{equation}
      \label{eq:collectK-mult-cv}
      \mult(\lambda, T_\infty M_\eta)= \sum_{\lambda'\in
      \spec(T_n M_\eta),\, |\lambda-\lambda'|\leq r} \mult(\lambda', T_nM_\eta).
  \end{equation}
  \end{enumerate}
\end{lemma}
\begin{proof}
  Set $T'_n=T M_{\eta v_n w_n}$ for $n\in \bar \N$,  where  $\bar  \N=\N
  \cup\{+\infty  \}$. 
  Using  Lemma~\ref{lem:prop-spec-mult}~\ref{item:spec-adjoint-mult}
  for the second equality,
  we have that for $n\in \bar \N$:
  \[
    \spec[T_n](\eta)=\spec(M_{v_n} T M_{\eta w_n})
    =\spec(T M_{\eta v_n w_n})=\spec(T'_n),
  \]
  and  similarly for  the  multiplicity.  Notice  the  set of  functions
  $\Delta'=\{ \eta v_n  w_n\, \colon\, \eta\in \Delta \text{  and } n\in
    \N\}$  is  bounded  in
  $L^\infty   $   and  thus   the   set   of  multiplication   operators
  $\{M_h\, \colon\, h\in \Delta'\}$ is bounded in $\cll(L^p)$. We deduce
  from      \cite[Proposition~4.2]{anselone}      that      the      set
  $\{TM_h\, \colon\, h\in \Delta'\}$ is collectively compact. In
  particular, the sequence $(T'_n, n\in  \N)$ is collectively compact. 

Let $h\in L^p$, we have $\norm{T'_\infty  h -T'_n h}_p\leq  \norm{T}_{L^p}
\norm{(v_\infty  w_\infty  -v_n w_n)h}_p$. Then, use dominated
convergence to get that $\lim_{n\rightarrow \infty } \norm{(v_\infty
  w_\infty  -v_n w_n)h}_p=0$. This implies that the sequence $(T'_n, n\in \N)$
  converges strongly to $T'_\infty $. 
Then use 
  Lemma~\ref{lem:prop-spec-mult}~\ref{item:collectK-cv} to
  conclude.
\end{proof}

\begin{remark}[On integral operators]\label{rem:op-noyau}
  Consider the  positive integral operator defined by:
  \begin{equation}
    \label{eq:def-op-int}
    T_\kk ( g)(x) = \int_\Omega \kk(x,y) \, g(y) \,\mathrm{d}\mu(y),
  \end{equation}
  where  $\kk$ is a  kernel on $\Omega$, that is a non-negative measurable function defined
  on $\Omega\times \Omega$. Under the hypothesis that  $\kk$ has a finite double norm in
  $L^p$ for some $p\in [1, +\infty )$, that is:
  \begin{equation}\label{eq:norm-pq}
    \norm{\kk}_{p,q}^p=\int_\Omega\left( \int_\Omega \abs{\kk(x,y)}^q\,
    \mu(\mathrm{d}y)\right)^{p/q} \mu(\mathrm{d}x)
  \end{equation}
  is finite with~$q=p/(p-1)$, the operator $T_\kk$ is  compact if $p>1$ and  $T_\kk^2$
  compact  if $p=1$; see~\cite[p.\ 293]{grobler}.   When $p>1$, one gets  stronger results
  on the  continuity  of  the  function $R_e[T_\kk]$;  see  Theorem  3.5  and Proposition 
  3.6  in~\cite{ddz-theo}   (where  $R_e[T_\kk]$  is  denoted $R_e[\kk]$ therein).
\end{remark}

We conclude this section with a remark on the definition of the operator 
$M_fTM_g$ when $T$ is a positive operator and $f$ and $g$ are
non-negative measurable function.

 \begin{remark}[On $M_f\, T\, M_g$]\label{rem:def-fTg}
   Let  $T$ be  a  positive  compact operator  on  $L^p(\mu)$ for  some $p\in [1, +\infty
   )$ and  $f,g$ be non-negative measurable functions defined on $\Omega$.  If the
   functions  $f, g$ are bounded, then the operator  $M_fTM_g$ is  a positive  compact
   operator  on $L^p(\mu)$. Motivated by  Example~\ref{ex:model=SIS2},
   we  shall however  be interested  in considering possibly unbounded functions $f$ and
   $g$. In this case, the operator $TM_g$ is      a     positive     compact      operator
   from $E=L^p((1+g)^p \, \rd  \mu)$ to $L^p(\mu)$, and thus  $M_fTM_g$ is a positive 
   compact       operator     from  $E$       to $E'=L^p((1+f)^{-1} \, \rd \mu)$.  Let
   $r\in [1, +\infty )$ and $\nu$ a  $\sigma$-finite  measure   mutually  absolutely 
   continuous  with $\mu$. Taking $F=E$  or $E'$, and using the compatibility between $F$ 
   and $L^r(\nu)$ given by Lemma~\ref{lem:Lp-compatible},   we  deduce  that there
   exists  at most  a unique  continuous  extension of  $M_fTM_g$ as  a bounded  operator
   on  $L^r(\nu)$,  which we  shall  still denote  by $M_fTM_g$.  By construction, this
   extension, when it exists, is also positive. However, let  us stress that it is not
   compact \emph{a priori}.
 \end{remark}

\section{Spectrum-preserving transformations}
\label{sec:equivalent}

In this section, we consider a  measured space $(\Omega, \cf, \mu)$ with
$\mu$ a non-zero $\sigma$-finite measure,  and we discuss two operations
on the positive compact operator $T$ which leave invariant the functions
$\spec[T]$ and  $R_e[T]$ defined on  $\Delta$. Recall the  discussion on
the operator $M_f T M_g$ from Remark~\ref{rem:def-fTg}.

\begin{lemma}\label{lem:hk/h}
     Let $T$ be a positive  compact operator on $L^{p}$ for some $p\in [1,
    +\infty )$ and $h$ be a
   measurable non-negative function defined on $\Omega$. 
 
  \begin{enumerate}[(i)]
  \item\label{lem:hk=kh} If $M_hT$ and $TM_h$ are positive compact
    operators (respectively on $L^r$ and $L^s$ with $r,s\in [1, +\infty
    )$ possibly distinct), then we have:
    \begin{align*}
      \spec[M_h T]= \spec[M_h T M_{\set{h>0}}] &=
      \spec[M_{{\set{h>0}}}\, T M_h] = \spec[TM_h],\\
      R_e[M_hT]=R_e[M_hTM_{{\set{h>0}}}] &=
      R_e[M_{{\set{h>0}} } \,TM_h]=R_e[TM_ h]. 
    \end{align*}
  \item\label{lem:k=hk/h} If $h$ is positive and if $M_h\, T \, M_{1/h}$
    is a  positive compact
    operators (on some  $L^r$  with $p,r\in [1, +\infty
    )$ possibly distinct),
     then we have:
     \[
       \spec[T]= \spec[M_h\, T \, M_{1/h}] \quad\text{and}\quad
      R_e[T]=R_e[M_h\, T \, M_{1/h}].
    \]
  \item\label{lem:k=ktop} The adjoint
    operator $T^\top$ is a   positive compact
    operator on $L^q$, with $q=p/(p-1)$ and  we have:
    \[
      \spec[T]=\spec[T^\top] \quad \text{and} \quad R_e[T] =
      R_e[T^\top].
    \]
  \end{enumerate} 
\end{lemma}

Let us stress   that the compactness hypothesis of $T$  can be removed in
the  statement  of~\ref{lem:hk=kh}.   Even   if  \ref{lem:k=hk/h}  is  a
consequence    of    \ref{lem:hk=kh},    we    state    it    separately
since~\ref{lem:k=hk/h}  and~\ref{lem:k=ktop} describe  two modifications
of~$T$  that  leave  the  functions~$R_e$  and  $\spec$  invariant.  See
Remark~\ref{rem:tilde-T} and Lemma~\ref{lem:Sp=DiagSym}
for   other  transformations on  the operators
which  leaves  the functions  $R_e$  and  $\spec$ invariant.   See  also
\cite{ddz-pres} for further results in the finite dimensional case.
 
\begin{proof}
  Since   $R_e=\mathrm{rad}\circ  \spec$,   we   only   need  to   prove
  \ref{lem:hk=kh}-\ref{lem:k=ktop}  for the  function $\spec$.   We give
  a   detailed  proof   of~\ref{lem:k=hk/h}   and   leave  the   proof
  of~\ref{lem:hk=kh}, which  is very  similar, to  the reader.  We first
  assume that $T$  is a positive compact operator on  $L^p$, and $h$ and
  $1/h$ are bounded.  The operators $TM_{\eta}$ and $M_h  T M_{\eta /h}$
  and  the  multiplication operators  $M_h$  and  $M_{1/h}$ are  bounded
  operators on $L^p$. We have, using  that $ T \, M_{\eta/h}$ is compact
  and~\eqref{eq:r(AB)-mult} for the second equality:
  \[
    \spec(T M_{\eta}) = \spec (T M_{ \eta/h}\, M_h) = \spec(M_h T M_{ \eta/h}).
  \]
  Since $\eta \in \Delta$ is arbitrary, this gives that
  $\spec[T]=\spec[M_h T M_{1 /h}]$.
  \medskip

  In the general case, we use  an approximation scheme.  Assume that $T$
  and  $T'=M_hTM_{1/h}$ are  positive  compact  operators (respectively  on
  $L^p$ and  $L^r$ with  $p,r\in [1,  +\infty )$) and  $h$ is  a positive
  function.  For $n\in \N^*$, set:
  \[
    v_n=\ind{\{n\geq h\geq 1/n\}}
    \quad\text{and}\quad
    h_n=n^{-1}
    \vee (h\wedge n).
  \]
  Notice that $T_n= M_{v_n} T M_{v_n}$ and $T''_n=M_{v_n h_n} T
  M_{v_n/h_n}$ are positive  compact  operators on $L^p$. 
   Let $\eta\in \Delta$. From the first part of the proof, we get that:
  \[
    \spec(T_n M_{\eta})=\spec(T''_n M_\eta).
  \]
  Consider also the   positive  compact  operator on $L^r$ defined by
  $T'_n=  M_{v_n} T'  M_{v_n}$.
  Since the sequence $(v_n, n\in \N^*$) converges in $L^\infty $ to $\un$, we
  deduce from Lemma~\ref{lem:cont-spec} that:
  \[
    \lim_{n\rightarrow \infty } \spec(T_n M_\eta)=\spec(T M_\eta)
    \quad\text{and}\quad
    \lim_{n\rightarrow \infty } \spec(T'_n M_\eta)=\spec(T' M_\eta). 
  \]
  Since $L^p $ and $L^r$ are compatible, see
  Lemma~\ref{lem:Lp-compatible} and the compact operators $T'_n$ and $T''_n$ are
  consistent,  we deduce from
Lemma~\ref{lem:prop-spec-mult}~\ref{item:density-mult} that:
\[
  \spec(T''_n M_\eta)=\spec(T'_n M_\eta).
\]
In conclusion, we obtain  that $  \spec(T M_\eta)=\spec(T' M_\eta)$ and
thus $  \spec[T]=\spec[T']$. 

 \medskip

 We now prove \ref{lem:k=ktop}. Notice that $TM_\eta$ and $(TM_\eta)^\top$
 are compact operator. 
 We have:
 \[
   \spec(T^\top M_\eta)=\spec(M_\eta T^\top)=
   \spec((TM_\eta)^\top)=\spec(T M_\eta), 
 \]
 where we used 
\eqref{eq:r(AB)-mult} for the first equality,
and~\eqref{eq:adjoint-mult} for the second. Since this is true for any
$\eta\in \Delta$, this gives   $\spec[T^\top] = \spec[T]$.
\end{proof}


\begin{remark}[Multiplicity of the eigenvalues]\label{rem:hk/h-mult}
  Following closely the proof of Lemma~\ref{lem:hk/h}~\ref{lem:k=hk/h},
  we also get  under the assumption of
  Lemma~\ref{lem:hk/h}~\ref{lem:k=hk/h} that:
  \begin{equation}\label{eq:mult:k=hk/h}
    \mult(\lambda, T)=\mult (\lambda, M_h T M_{1/h}) \quad\text{ for all $\lambda\in
    \C^*$.}
  \end{equation}
\end{remark}

\section{Sufficient conditions for convexity or concavity
  of~\texorpdfstring{$R_e$}{Re}}\label{sec:Hill_Longini}

\subsection{A conjecture from Hill and Longini}
\label{sec:comparison_Hill_Longini}

Recall  that,  in the  metapopulation  framework  with $N$  groups,  the
effective reproduction  number is  equal to the  spectral radius  of the
matrix $K\cdot\Diag(\eta)$,  where the  next-generation matrix $K$  is a
$N\times     N$     matrix     with     non-negative     entries     and
$\eta\in  \Delta=[0,  1]^N$  is  the  vaccination  strategy  giving  the
proportion  of non-vaccinated  people in  each groups.   Hill and Longini
conjecture in~\cite{hill-longini-2003}  conditions on
the spectrum of the next-generation matrix that should imply  convexity or
 concavity of  the effective reproduction number. The conjecture  states that the
function $R_e[K]$ is:
\begin{enumerate}[(i)]
\item\label{item:cvx-K} convex when $\spec(K) \subset \R_+$,
\item\label{item:cave-K} concave when $\spec(K) \backslash \{ R_0 \} \subset \R_-$.
\end{enumerate}

It turns out that the conjecture cannot be true without additional assumptions on the
matrix~$K$. Indeed, consider the following next-generation matrix:
\begin{equation}\label{eq:next-gen-counter-conv}
  K =
  \begin{pmatrix}
    16 & 12 & 11 \\
    1 & 12 & 12 \\
    8 & 1 & 1
  \end{pmatrix}.
\end{equation}
Its eigenvalues are approximately equal to $24.8$, $2.9$ and $1.3$. Since $R_e$ is
homogeneous, the function is entirely determined by the value it takes on the plane $\{
\eta \, \colon \, \eta_1 + \eta_2 + \eta_3 = 1/3 \}$. The graph of the function $R_e$
restricted to this set has been represented in Figure~\ref{fig:counter-graph-conv}. The
view clearly shows the saddle nature of the surface. Hence, the Hill-Longini conjecture
\ref{item:cvx-K} is contradicted in its original formulation. 

In the same manner, the eigenvalues of the following next-generation matrix:
\begin{equation}\label{eq:next-gen-counter-conc}
  K = 
  \begin{pmatrix}
    9 & 13 & 14 \\
    18 & 6 & 5 \\
    1 & 6 & 6
  \end{pmatrix}
\end{equation}
are approximately equal to $26.3$, $-1.4$ and $-3.9$. Thus, $K$ satisfies the condition
that should imply the concavity of the effective reproduction number in the Hill-Longini
conjecture~\ref{item:cave-K}. However, as we can see in
Figure~\ref{fig:counter-graph-conc}, the function $R_e$ is neither convex nor concave.

\begin{figure}
  \begin{subfigure}[T]{.5\textwidth} \centering
    \includegraphics[page=3]{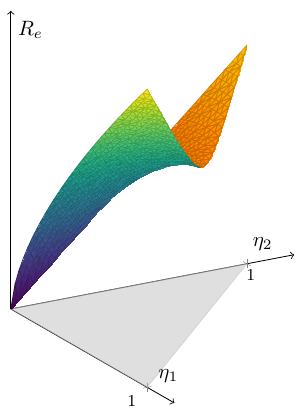}
    \caption{Convex case: $K$ given by~\eqref{eq:next-gen-counter-conv}.}
    \label{fig:counter-graph-conv}
  \end{subfigure}%
  \begin{subfigure}[T]{.5\textwidth} \centering
    \includegraphics[page=4]{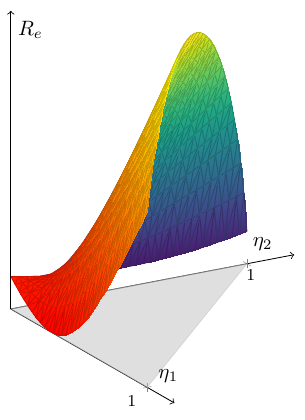}
    \caption{Concave case: $K$ given by~\eqref{eq:next-gen-counter-conc}.}
  \label{fig:counter-graph-conc}
  \end{subfigure}%
  \caption{Counter-example of the Hill-Longini  conjecture.  The plan of
    strategies    $P=\{\eta\,   \colon\,    \eta_1+\eta_2+\eta_3=1/3\}$   is
    represented as a gray  surface. The triangulated surface corresponds
    to  the  graph  of  $\eta\mapsto R_e[K](\eta)$  restricted to $P$.}
  \label{fig:counter-conc}
\end{figure}

Despite these  counter-examples, the  Hill-Longini conjecture  is indeed
true   when   making   further   assumption   on   the   next-generation
matrix.   Let~$M$  be   a  square   real  matrix.   The  matrix~$M$   is
\emph{diagonally  similar}  to  a  matrix~$M'$ if  there  exists  a  non
singular  real   diagonal  matrix~$D$   such  that~$M=D\cdot   M'  \cdot
D^{-1}$. The matrix~$M$ is said to be \emph{diagonally symmetrizable} or
simply \emph{symmetrizable} if  it is diagonally similar  to a symmetric
matrix, or,  equivalently, if $M$  admits a decomposition~$M=D  \cdot S$
(or~$M=S\cdot D$), where~$D$ is a diagonal matrix with positive diagonal
entries and~$S$  is a symmetric  matrix. If  a matrix $M$  is diagonally
symmetrizable,  then its  eigenvalues  are real  since similar  matrices
share  the  same spectrum.  We  obtain  the  following result  when  the
next-generation  matrix  is  symmetrizable   as  a  particular  case  of
Theorem~\ref{th:hill-longini} below.

\begin{theorem}\label{th:hill-longini-meta}
  Let  $K$  be  a  diagonally  symmetrizable  $N\times  N$  matrix  with
  non-negative entries,  and consider the function  $R_e=R_e[K]$ defined
  on $\Delta=[0, 1]^N$.
  \begin{enumerate}[(i)]
  \item\label{pt:conv-meta} If $\spec(K) \subset \R_+$, then the
    function $R_e$ is convex.
  \item If $R_0=R_0[K]$ is a simple eigenvalue of $K$ and
    $\spec(K) \subset \R_-\cup \{ R_0 \} $, then the function
    $R_e$ is concave.
  \end{enumerate}
\end{theorem}

 The first
point~\ref{pt:conv-meta} has been proved by Cairns in \cite{EpidemicsInHeCairns1989}. In
\cite{Fried80}, Friedland obtained that, if the next-generation matrix $K$ is not singular
and if its inverse is an M-matrix (\emph{i.e.}, its non-diagonal coefficients are
non-positive), then $R_e[K]$ is convex. Friedland's condition does not imply that $K$ is
symmetrizable nor that $\spec(K) \subset \R_+$. On the other hand, the following matrix is
symmetric definite positive (and thus $R_e$ is convex) but its inverse is not an M-matrix:
\begin{equation*}
  K = 
  \begin{pmatrix}
    3 & 2 & 0 \\
    2 & 2 & 1 \\
    0 & 1 & 4
  \end{pmatrix}
  \qquad \text{with inverse} \quad
  K^{-1} = 
  \begin{pmatrix}
    1.4 & -1.6 & 0.4 \\
    -1.6 & 2.4 & -0.6 \\
    0.4 & - 0.6 & 0.4 
  \end{pmatrix}.
\end{equation*}
Thus Friedland's condition and Property~\ref{pt:conv-meta} in
Theorem~\ref{th:hill-longini-meta} are not comparable. Note that if~$K$
is diagonally symmetrizable
and its inverse is an M-matrix, then the eigenvalues of $K$ are actually non-negative
thanks to \cite[Chapter~6~Theorem~2.3]{NonnegativeMatBerman1994} and one can apply
Theorem~\ref{th:hill-longini-meta}~\ref{pt:conv-meta} to recover Friedland's result  in this
case.

\subsection{Generalization to compact operators}

In this section, we give the analogue of
Theorem~\ref{th:hill-longini-meta} for positive compact operators 
instead of matrices. First, we proceed with some definitions.
By  analogy   with  the  matrix  case,   we   introduce    the   notion   of   diagonally
symmetrizable operators. 

Recall that $(\Omega, \cf, \mu)$ is a measured space with $\mu$ a
$\sigma$-finite non-zero measure. 
Recall the definition of consistent operators given in
Section~\ref{sec:inv-cont} before Lemma~\ref{lem:prop-spec-mult}. 
\begin{definition}[Diagonally symmetrizable operator]
  \label{def:diag-sym}
 A compact  operator $T$ on  $L^p(\mu)$, with  $p\in [1, +\infty  )$, is
  called \emph{diagonally symmetrizable} if there exists a
  $\sigma$-finite measure $\mu'$
  mutually absolutely continuous with respect  to $\mu$, and a compact
  self-adjoint operator
  $T'$ on $L^2(\mu')$ such that $T$ and
  $T'$ are consistent.
\end{definition}  

\begin{remark}[Diagonally symmetrizable operator in finite dimension]\label{rem:op=mat}
  Let  us check  that  Definition~\ref{def:diag-sym}  coincide with  the
  definition of  diagonally symmetrizable matrices in  finite dimension.
  Let $\Omega$ be a finite set, say $\{1, \ldots, n\}$, and without loss
  of generality assume that the measure  $\mu$, as well as $\mu'$, which
  can be seen as vectors of $\R^n$,  have  positive entries. The  sets
  $L^p(\mu)$ and $L^2(\mu')$  are all equal to $\R^n$ and  $T=T'$ can be
  represented by a  matrix, say $M$, in the canonical  base of $\R^n$. Let
  $D$ be the diagonal matrix with diagonal entries $\mu'$. Then $T'$
  being self-adjoint in $L^2(\mu')$ is equivalent to $DM$ being
  symmetric, and thus the matrix $M$ is diagonally symmetrizable (in the
  sense of the previous section). 
\end{remark}

We  give  an  example  of  diagonally  symmetrizable  integral  operator
motivated      by       the      epidemiological       framework      of
Example~\ref{ex:model=SIS2}         below.           Recall         from
Remark~\ref{rem:op-noyau}   that  a   kernel  $k$   on  $\Omega$   is  a
non-negative measurable function defined on $\Omega\times \Omega$.

\begin{proposition}[Diagonally symmetrizable integral operators]
  \label{prop:diag-kernel}
  Let $p\in (1, +\infty )$ and $q$ its conjugate, $k$ be a symmetric kernel on $\Omega^2$, and
  $f,g$ be two positive measurable functions on $\Omega$ such that:
  \begin{align*}
    \int_\Omega  f(x)^p \, \left( \int_\Omega k(x,y)^q\, g(y)^q\,  \mu(\rd y)
     \right)^{p/q}d\mu(x) &<+\infty,  \\ 
    \int_{\Omega^2}f(x)g(x)\,  k^2(x,y) \, f(y)g(y)\  \mu(\rd x) \mu(\rd
    y) &<+\infty .
\end{align*}
Then,                the                integral                operator
$T:u\mapsto \left(x\mapsto \int_\Omega f(x)\,  k(x,y)\, g(y)u(y)\, \mu(\rd y)\right)
$ on $L^p(\mu)$ is compact positive and diagonally symmetrizable.
\end{proposition}

\begin{proof}
  The measure $\rd \mu' = (g/f) \, \rd \mu$ is $\sigma$-finite and mutually
  absolutely  continuous with  respect to  $\mu$. Consider  the integral
  operator $T':u \mapsto \left(x \mapsto  \int_\Omega f(x)\, k(x,y)\, f(y) u(y)\,
  \mu'(\rd y)\right)$.   The integrability assumptions ensure  that $T$
  is  compact  on   $L^p(\mu)$  and  $T'$  is  compact   (and  in  fact,
  Hilbert-Schmidt)   on   $L^2(\mu')$; see Remark~\ref{rem:op-noyau}.
  According   to   Lemma~\ref{lem:Lp-compatible},  the   Banach   spaces
  $L^p(\mu)$ and $L^2(\mu')$ are compatible.  Since the operators $T$ and
  $T'$ are defined by the same kernel formula on their respective space,
  they  are consistent.   Finally  since the  compact  operator $T'$  is
  clearly self-adjoint on $L^2(\mu')$. This implies that $T$ is diagonally symmetrizable.
  \end{proof}

\begin{example}[Epidemics on graphon]\label{ex:model=SIS2}
  Consider the SIS model on graphon
  introduced in \cite[Example~1.3]{delmas_infinite-dimensional_2020}. In this example, the
  next-generation operator is an integral operator as defined in
  Remark~\ref{rem:op-noyau} associated to the kernel $\kk$ given by $\kk(x,y)=
  \beta(x)\, W(x,y)\, \theta(y)/\gamma(y)$ where $\beta(x)$ represents the
  susceptibility,  $\theta(x)$ the infectiousness and $\gamma(x)$ the recovery rate of the
  individuals with trait  $x$, and $W$ corresponds to the graph  of the contacts within 
  the population. More precisely, for $x,y \in \Omega$, the  quantity $W(x,  y) \in [0,
  1]$ represents the  density of contacts between  individuals with traits $x$ and $y$
  and is equal to $W(y,x)$ by construction. We deduce from
  Proposition~\ref{prop:diag-kernel} that if $\beta\in L^p(\mu)$ and $\theta/\gamma\in
  L^q(\mu)$ with $p\in (1, +\infty )$ and $q$ its conjugate, then the integral operator
  $T_\kk$ with kernel $\kk$ defined by~\eqref{eq:def-op-int} is  diagonally 
  symmetrizable. 
\end{example}

\begin{remark}[Related notions]
    An operator $T$ on a Hilbert space is classically called symmetrizable
    if there exists a positive bounded self-adjoint operator $H$ such that $HT$
    is self-adjoint: this notion is discussed for example in
    \cite{Rei51,zaanen_linear,Har71}.
    Our definition is closer in spirit to~\cite{Lax54}, where symmetrizability
    is discussed for operators on Banach spaces with respect to a scalar product.
    In the matrix case our setting is a bit more restrictive than general symmetrizability,
    since we symmetrize by a diagonal matrix with positive terms. In the
    general case the conditions are not comparable, since we do not
    impose any upper nor lower bound  assumption on the density $d\nu/d\mu$. 
\end{remark}

We complete  Section~\ref{sec:equivalent} with  an  other example of
operators having the same effective spectrum.
  
\begin{lemma}\label{lem:Sp=DiagSym}
  Let $T$ be a diagonally symmetrizable compact operator on $L^p(\mu)$, with
  $p\in  [1, +\infty  )$,  and let $T'$ be the associated self-adjoint operator from
  Definition~\ref{def:diag-sym}. Then, we have that on $\Delta$:
  \[
    \spec[T]=\spec[T'], \quad R_e[T]=R_e[T']
    \quad\text{and}\quad
    \, \mult(\lambda, T)=\mult(\lambda, T')\quad\text{for $\lambda\in \C^*$}.
  \]
\end{lemma}

\begin{proof}
  Let $\mu'$ be the measure from Definition~\ref{def:diag-sym}. Recall that the Banach
  spaces $L^p(\mu)$ and on  $L^2(\mu')$ are compatible thanks to
  Lemma~\ref{lem:Lp-compatible}.  Let $\eta\in \Delta$. Since $M_\eta$ is bounded  (both
  on $L^p(\mu)$ and  $L^2(\mu')$),  the operators  $TM_\eta$  and  $T M_\eta$,  acting
  respectively on  $L^p(\mu)$ and  $L^2(\mu')$, are both  compact. Since $T$ and  $T'$ are
  consistent, the operators $TM_\eta$ and  $T' M_\eta$  are also consistent. Then use
  Lemma~\ref{lem:prop-spec-mult}~\ref{item:density-mult} to conclude. 
\end{proof}

The next corollary is immediate as the spectrum of a self-adjoint
operator, say $T'$,  is real and its spectral radius is zero if and only
if $T'=0$. 

\begin{corollary}
  \label{cor:diag-sym=spec-reel}
  Let $T$ be a compact operator  on $L^p(\mu)$,
  with $p\in  [1, +\infty )$. If  $T$  is diagonally symmetrizable, then  its spectrum is real,
  and $T$ cannot be quasi-nilpotent: 
    $R_0(T)=0$ if and only if  $T=0$.
\end{corollary}

For  a  compact  operator  $T$,  let  $\pos(T)$  and  $\nega(T)$  denote
respectively the  number of its  eigenvalues with positive  and negative
real part taking into account their (algebraic) multiplicity:
\[
  \pos(T)= \sum_{\Re(\lambda)>0} \mult(\lambda,T)
  \quad\text{and}\quad
  \nega(T)= \sum_{\Re(\lambda)<0} \mult(\lambda,T). 
\]

We now give a consequence of Sylvester's inertia theorem
\cite[Theorem 6.1]{Cain80}.

\begin{proposition}[Sylvester]
  \label{prop:sylvester}
  Let $T$ be a compact  diagonally symmetrizable operator on $L^p(\mu)$,
  with $p\in [1, +\infty  )$. Let $f,  g$ be positive bounded
  measurable functions defined on~$\Omega$  which are also bounded
  away from $0$. Then the compact operator $M_f T M_g$ on $L^p(\mu)$ is
  diagonally symmetrizable with the same inertia as $T$:
  \[
    \pos(T)=\pos(M_f T M_g)
    \quad\text{and}\quad
    \nega(T)=\nega(M_f T M_g).
  \]
\end{proposition}
\begin{proof}
  First note that if $h$ is a positive bounded and bounded
  away from $0$, then for any $r\in[1, +\infty )$ and $\sigma$-finite
  non-zero measure $\nu$, the multiplication operator
  $M_h$ is bounded with bounded inverse on $L^r(\nu)$.
  In particular, the operator $\tilde{T} = M_f TM_g$ is a compact
  operator on  $L^p(\mu)$ as $T$ is compact. 

  Let $T'$ be the compact self-adjoint operator on $L^2(\mu')$ associated to $T$
  from         Definition~\ref{def:diag-sym}.        The         measure
  $\rd  \tilde{\mu}' =  (g/f)\, \rd  \mu'$ is  $\sigma$-finite and  mutually
  absolutely continuous with  respect to $\mu'$ and   $\mu$ also.  The mapping
  $\Phi=M_{\sqrt{f/g}}$ is an isometry between the Hilbert spaces $L^2(\tilde \mu')$
  and~$L^2(\mu')$.

  We now define the operator $\tilde{T}'$ on $L^2(\tilde \mu')$ by:
  \[
    \tilde{T}' = \Phi \circ (M_{\sqrt{fg}}T' M_{\sqrt{fg}}) \circ
    \Phi^{-1}.
  \]
  Since  $T'$ is  compact, the  operator $\tilde{T}'$  is also  compact.
  Since  $f$ and  $g$ are  bounded  and bounded  away from  0, the  sets
  $L^p(\mu)\cap  L^2(\mu')$  and   $L^p(\mu)\cap  L^2(\tilde{\mu})$  are
  equal.    Since  $T$   and  $T'$   coincide   on  this   set,  so   do
  $M_fTM_g    =    \tilde{T}$    and   $\tilde{T}'$.     The    operator
  $M_{\sqrt{fg}}  T'   M_{\sqrt{fg}}$  is   bounded  and   symmetric  on
  $L^2(\mu')$, and therefore self-adjoint.  Since $\Phi$ is an isometry,
  we deduce  that $\tilde{T}'$  is self-adjoint  on $L^2(\tilde{\mu}')$.
  Therefore  the  operator  $\tilde{T}  = M_f  TM_g$  on  $L^p(\mu)$  is
  diagonally symmetrizable. 

  We now establish the following string of equalities:    
  \begin{equation}
    \label{eq:stringOfPos}
    \pos(T)=\pos(T')=\pos(M_{\sqrt{fg}}T' M_{\sqrt{fg}})
    =\pos(\tilde T')=\pos(\tilde T)=\pos(M_f TM_g),
  \end{equation}
  By Lemma~\ref{lem:Sp=DiagSym},
 $ \pos(T) = \pos(T')$ and $\pos(\tilde{T}') = \pos(\tilde{T})$. 
  Since $M_{\sqrt{fg}}$ is invertible in $L^2(\mu')$ and $T'$ is
  self-adjoint (thus with real eigenvalues), we get, using the
  generalization of Sylvester's inertia theorem
  \cite[Theorem~6.1]{Cain80} (the definition of inertia in that
  paper being consistent with the definition of $\pos(\cdot)$ and $\nega(\cdot)$,
  which can be checked using~\cite[Theorem~4.5(ii)]{Cain80}):
  \( \pos(T') = \pos(M_{\sqrt{fg}} T' M_{\sqrt{fg}})\).
  Finally, since $\Phi$ is an isometry,
  \(\pos(M_{\sqrt{fg}} T' M_{\sqrt{fg}}) = \pos(\tilde{T}')\),
  and~\eqref{eq:stringOfPos} is justified.

  The equalities are similar for the
  number of negative eigenvalues $\nega(\cdot)$.
\end{proof}

The following result is the analogue of Theorem~\ref{th:hill-longini-meta} for positive
compact operators. Note that  if $T$ is a positive compact operator with $R_0[T]>0$, then
$R_0[T]$ is an eigenvalue of  $T$ thanks to the Krein-Rutman theorem, see
\cite[Corollary~9]{schwartz61}, and thus $\pos(T)\geq 1$.

\begin{theorem} [Convexity/Concavity of $R_e$]\label{th:hill-longini}
  Let $T$  be a  positive compact  diagonally symmetrizable  operator on
  $L^p(\mu)$,  with  $p\in [1,  +\infty  )$.  We consider  the  function
  $R_e=R_e[T]$ defined on $\Delta$.
\begin{enumerate}[(i)]
   \item\label{item:hill-longini-cvx} If $\nega(T)=0$, then the function $R_e$
     is convex.
   \item\label{item:hill-longini-cave} If
     $\pos(T)=1$, 
     then the function $R_e$ is concave.
   \end{enumerate}
 \end{theorem}
 The proof  for positive self-adjoint  operator $T$ is given  in Section
 \ref{sec:convexe} for the convex  case and in Section \ref{sec:concave}
 for the concave case when $T$  is compact.  The extension to diagonally
 symmetrizable   positive  compact   operators  follows   directly  from
 Lemma~\ref{lem:Sp=DiagSym}.

\begin{remark}[Rank-one operator]\label{rem:config}
  The so-called \emph{configuration  model} occurs in finite dimension  when the next-generation 
  matrix has rank  one.  This corresponds to a classical mixing  structure called the
  \emph{proportionate mixing} introduced by \cite{HeterogeneityINold1980} and used in many
  different epidemiological models. Motivated by the  finite dimensional case, we consider
  a \emph{configuration} kernel $\kk$ defined by
  \[
    \kk=f\otimes  g \quad\text{where}  \quad (f\otimes g)  (x,y) =  f(x) g(y),
  \]
  with  $f\in  L^p$  and  $g  \in   L^q$  for some $p \in (1, +\infty)$ and  $q =
  p/(p-1)$.  We also suppose  that $\mu(fg  > 0)  > 0$.   Let $T_\kk$  denote the 
  integral operator with kernel  $\kk$, see Remark~\ref{rem:op-noyau}. According to       
  Proposition~\ref{prop:diag-kernel},              with $k=\ind{\{f>0\}} \otimes 
  \ind{\{g>0\}}$ and $h\in \{f,  g\}$ replaced by     $h+h'\ind{\{h=0\}}$     for     
  some     positive     function $h'\in L^p(\mu)\cap  L^q(\mu)$, we  deduce that the 
  integral operator $T_\kk$   on   $L^p(\mu)$   is   compact   positive   and   diagonally
  symmetrizable. Since $T_\kk$ is of rank one,  we deduce from
  Theorem~\ref{th:hill-longini} that  $R_e[T_\kk]$  is convex and concave and thus linear.
  This can be checked directly as it is immediate to notice that:
  \begin{equation}
    R_e[T_\kk](\eta) = \int_\Omega f g\, \eta \, \mathrm{d}\mu.
  \end{equation}
  We shall provide in \cite{ddz-mono} a deeper study of configuration kernels in the
  context of epidemiology.
\end{remark}
  
\subsection{The convex case}\label{sec:convexe}

The      proof      of      Property~\ref{item:hill-longini-cvx}      in
Theorem~\ref{th:hill-longini} relies on an idea from \cite{Fried80} (see
therein  just  before Theorem  4.3).   Let  $T$  be a
self-adjoint      operator     on      $L^2=L^2(\mu)$     such      that
$  \spec(T)\subset   \R_+$.   As  $R_0[T]=0$  implies   $T=0$  and  thus
$R_e[T]=0$, we shall only consider the case $R_0[T]>0$.
Since $T$ is a self-adjoint positive semi-definite operator on $L^2$, there exists a
self-adjoint positive semi-definite operator $Q$ on $L^2$ such that $Q^2=T$. Thanks to \eqref{eq:r(AB)}, we have for $\eta\in \Delta$:
\[
  R_e[T](\eta)=\rho(T \,M_\eta)=\rho(Q^2\,
  M_\eta)= \rho ( Q \,M_\eta\, Q). 
\]
Since the self-adjoint operator $Q\, M_{\eta}\, Q$ on $L^2$ is also positive
semi-definite, we deduce from the Courant-Fischer-Weyl min-max principle that:
\[
  R_e[T](\eta)= \rho\left(Q\, M_{\eta}\, Q\right)=\sup_{u\in L^2(\mu)\setminus\set{0}}
  \frac{\langle u, Q \, M_{\eta} \,Qu \rangle}{ \langle u,u \rangle}\cdot
\]
Since the map $\eta \mapsto \langle u, Q \, M_{\eta} \,Q u \rangle$ defined on $\Delta$
is linear, we deduce that $\eta\mapsto R_e[T](\eta)$ is convex as a supremum of linear
functions.

\subsection{The concave case}\label{sec:concave}

The proof of Property~\ref{item:hill-longini-cave} in
Theorem~\ref{th:hill-longini}  relies on  a  computation  of the  second derivative  of 
the function  $R_e$.   Let  $T$  be a  positive  compact self-adjoint operator on
$L^2(\mu)$ such that $\pos(T)=1$.  Let $\Delta^*$ be the subset of  $\Delta$ of the
functions which are  bounded away from 0. The  set $\Delta^*$  is  a  dense convex  subset
of  $\Delta$ (for  the $L^2(\mu)$-convergence or simple convergence).   The function
$R_e=R_e[T]$ is continuous on $\Delta$, see Lemma~\ref{lem:cont-spec} (indeed, with the
notations  therein,  take  $v_n=1$,  $w_n\in  \Delta$  and  notice  that
$R_e[T_n](\un)=R_e[T](w_n)$                 converges                 to $R_e[T_\infty
](\un)=R_e[T](w_\infty  )$). So its suffice  to prove that $R_e=R_e[T]$  is  concave  on 
$\Delta^*$.  Let  $\eta_0$,  $\eta_1$  be elements           of            $\Delta^*$,    
and           set $\eta_\alpha =  (1-\alpha)\eta_0 + \alpha\eta_1$ for  $\alpha\in
[0, 1]$ (which    is   also    an    element   of    $\Delta^*$).    We    write $T_\alpha
= T  M_{\eta_\alpha}$, so that $T_\alpha = T_0  + \alpha T M$, where    $M=M_{\eta_1   
-\eta_0}$     is    the    multiplication    by $(\eta_1 - \eta_0)$ operator, and, with
$R(\alpha) = R_e(\eta_\alpha)$:
\[
  R(\alpha) = \rho(T_\alpha)=\rho(T_0 + \alpha T
  M). 
\]
So, to prove that $R_e$ is concave on $\Delta^*$ (and thus on $\Delta)$, it is enough to
prove that $\alpha\mapsto R(\alpha)$ is concave on $(0, 1)$.
Thanks to Sylvester's inertia theorem stated in Proposition~\ref{prop:sylvester} (with
$f=1$ and $g=\eta_\alpha$), we also get that $\pos(T_\alpha)=\pos(T)=1$. This implies that
$R(\alpha)$ is positive and a simple eigenvalue.

 We  consider the  following  scalar product  on  $L^2(\mu)$ defined  by
 $\braket{u,v}_\alpha   =   \braket{u,\eta_\alpha  v}$.   The   operator
 $T_\alpha$ is  self-adjoint and  compact on $L^2(\eta_\alpha  \rd \mu)$
 with         spectrum        $\spec(T_\alpha)$         thanks        to
 Lemma~\ref{lem:prop-spec-mult}~\ref{item:density-mult}              and
 Lemma~\ref{lem:Lp-compatible}.  Let $(\lambda_n, n\in I=\lb 0, N \lb)$,
 with  $N\in \N  \cup\{\infty  \}$  be an  enumeration  of the  non-zero
 eigenvalues   of   $T_\alpha$   with   their   multiplicity   so   that
 $\lambda_0=R(\alpha)>0$      and       thus      $\lambda_n<0$      for
 $n\in  I^*=I\setminus  \{0\}$;   and  denote  by  $(u_n,   n\in  I)$  a
 corresponding     sequence    of     orthogonal    eigenvectors     (in
 $L^2(\eta_\alpha   \rd  \mu)$).   The   functions  $v_\alpha=u_0$   and
 $\phi_\alpha=\eta_\alpha u_0$  are the right and  left-eigenvectors for
 $T_\alpha$  (seen   as  an   operator  on  $L^2(\mu)$)   associated  to
 $R(\alpha)$.

We now follow \cite{EffectivePertuBenoit} to get that
$\alpha \mapsto R(\alpha)=\rho(T_0 + \alpha T M)$ is analytic and
compute its second derivative. Let $\pi_\alpha$ be the projection on
the ($\braket{\cdot,\cdot}_\alpha$)-orthogonal of $v_\alpha$, and
define:
\[
  S_\alpha = (T_\alpha - R(\alpha))^{-1} \pi_\alpha.
\]
In other words, $S_\alpha$ maps $u_0$ to $0$ and $u_i$ to
$(\lambda_i - R(\alpha))^{-1}\, u_i$.
Let $\alpha\in(0,1)$ and $\varepsilon$ small enough so that
$\alpha+\varepsilon\in [0, 1]$. We have:
\[
    T_{\alpha+\varepsilon} = T_\alpha + \varepsilon T M,
  \]
  and thus $\norm{T_{\alpha+\varepsilon}
      -T_\alpha}_{L^2(\eta_\alpha \rm{d} \mu)}=O( \varepsilon)$.
    Using \cite[Theorem~2.6]{EffectivePertuBenoit} on the Banach space
    $L^2(\eta_\alpha\, \mathrm{d}\mu)$, we get that:
   \[
      R(\alpha + \varepsilon)
      = R(\alpha) + \varepsilon\braket{v_\alpha,T M v_\alpha}_\alpha
  - \varepsilon^2\braket{v_\alpha, T M S_\alpha T M v_\alpha}_\alpha
  + O(\varepsilon^3).
\]
Let $N_\alpha=M_{1/\eta_\alpha}M=MM_{1/\eta_\alpha}$ be the
multiplication by $(\eta_1-\eta_0)/\eta_\alpha$ bounded operator. 
Since $\alpha \mapsto R(\alpha)$ is analytic and $T$ is self-adjoint
(with respect to $\langle \cdot, \cdot \rangle$), we get that:
\begin{align*}
    R''(\alpha)
    &= - 2 \braket{v_\alpha,T M S_\alpha T M v_\alpha}_\alpha \\
    &= - 2 \braket{M T_\alpha v_\alpha, S_\alpha T M v_\alpha} \\
    &= - 2 R(\alpha) \braket{ M v_\alpha, S_\alpha T M v_\alpha}\\
    &= - 2 R(\alpha) \braket{ N_\alpha v_\alpha, S_\alpha T_\alpha N_\alpha v_\alpha}_\alpha.
  \end{align*}
  Since the kernel and the image of $T_\alpha$ are orthogonal (in
  $L^2(\eta_\alpha \rd \mu)$), and the latter is generated by
  $(u_n, n\in I)$, we have the decomposition
  $N_\alpha v_\alpha = g+ \sum_{n\in I} a_nu_n$ with
  $g\in \mathrm{Ker}(T_\alpha)$ and
  $a_n=\langle N_\alpha v_\alpha, u_n \rangle_\alpha/\langle u_n, u_n
  \rangle_\alpha$. This gives, with 
  $I^*=I\setminus \{0\}$:
\begin{equation}\label{eq:R2}
    R''(\alpha) = 2 R(\alpha) \sum_{n\in I^*} \frac{\lambda_n}{R(\alpha) - \lambda_n} \, a_n^2\, \braket{u_n,u_n}_\alpha. 
\end{equation}
Since $\lambda_n<0$ for all $n\in I^*$, we deduce that $R''(\alpha)\leq
0$ and thus $\alpha \mapsto R(\alpha)$ is 
concave on $[0, 1]$. This implies that $R_e[T]$ is concave. 

\begin{remark}\label{rem:cvxe2}
  The same proof with obvious changes gives that if $T$ is a positive
  quasi-irreducible compact self-adjoint operator 
  (see Section~\ref{sec:q-irr-monat} for the precise definition of 
  quasi-irreducible operator) such that $\nega(T)=0$, then $R_e[T]$
  is convex on $\Delta$. Then, using the decomposition of a compact
  operator on its irreducible atoms, see Section~\ref{sec:atomic} and
  more precisely Lemma~\ref{lem:Rei}, and since the maximum of convex
  functions is convex (used in~\eqref{eq:R=maxRi}), we can recover
  Theorem~\ref{th:hill-longini}~\ref{item:hill-longini-cvx}. 
\end{remark}

\section{The reproduction number and reducible positive compact operators}
\label{sec:reducible}

Following \cite{schwartz61},  we  present in Section~\ref{sec:atomic} the  atomic decomposition
of a positive compact operator $T$ on $L^p$ where $p\in [1, +\infty )$ and
state a formula which ``reduces''  the effective
reproduction function  of $T$ on the whole space to the ones of  the
restriction of $T$ to each atoms (or irreducible components); see Corollary~\ref{cor:Speci} below.
Then, we consider the notion of  quasi-irreducible and
monatomic operators in Section~\ref{sec:q-irr-monat}, and provide some
properties of monatomic operator and prove that if the effective
reproduction number is concave then the operator is monatomic. 

\subsection{Atomic decomposition} 
\label{sec:atomic}

Our  presentation is  a direct  application of  the Frobenius
decomposition, see \cite{victory82,  victory93} and \cite{schwartz61} or
the   ``super  diagonal''   form;  see   \cite[Part~II.2]{dowson}.   For
convenience, we  follow \cite{schwartz61} for positive  compact operator
on   $L^p(\mu)$   for   some   $p\in    [1,   +\infty   )$,   see   also
\cite[Lemma~5.17]{bjr} in the case  of integral operators with symmetric
kernel. We stress that the results in \cite{schwartz61} are stated under
the hypothesis that  $\mu$ is a finite measure, but  it is elementary to
check the  main results  (Theorems 7  and 8 therein) also  hold if  the measure
$\mu$ is $\sigma$-finite.

For  $A, B\in \cf$, we  write $A\subset B$
a.e.\ if $\mu(B^c  \cap A)=0$ and $A=B$ a.e.\ if  $A\subset B$ a.e.\ and
$B\subset A$ a.e..  Let $T$ be  a positive compact  operator on $L^p$ for
some  $p\in   [1,  +\infty   )$.  Let $f_0\in L^p$ and $g_0\in L^q$ be
positive functions and consider the operator $T_0=M_{g_0} T M_{f_0}$ from
$L^\infty $ to $L^1$. We  define   the  function   $\kk_T$  on
$\cf^{ 2}$ as, for $A, B\in \cf$:
\begin{equation}
   \label{eq:def-kT}
  \kk_T(B, A)=\int_B (T _0 \, \ind A)(x) \, \mu(\rd x)=
  \langle  \ind{B}, T_0\, 
   \ind{A} \rangle. 
\end{equation}
It is clear from~\eqref{eq:def-kT} that  the family of sets $(B,A)$ such
that  $\kk_T(B, A)=0$  does not  depend on  the choice  of the  positive
functions $f_0\in L^p$ and $g_0\in L^q$. If the measure $\mu$ is finite,
then one can take $f_0=g_0=\un$ and thus $T_0=T$.

A  set $A\in  \cf$ is  \emph{$T$-invariant}, or  simply \emph{invariant}
when there  is no ambiguity on  the operator $T$, if  $\kk_T(A^c, A)=0$.
We   recall   that   $\ci$   is   a  closed   ideal   of   $L^p$   (with
$p\in   [1,   +\infty   )$)   if   and  only   if   it   is   equal   to
$\ci_B=\{f\in  L^p\,  \colon\,  f\ind{B}=0\}$ for  some  measurable  set
$B\in \cf$,  see Example~2 in~\cite[Section~III.1]{schaefer_banach_1974}
or  \cite[Section~III.2]{zerner_87}.  Notice  that a  set $A\in  \cf$ is
invariant if  and only if the  ideal $\ci_A$ is invariant  for $T$, that
is, $T(\ci_A)\subset \ci_A$.

A positive  compact operator $T$ on  $L^p$ is (ideal)-\emph{irreducible}
if the  only closed invariant  ideal are  $\{0\}$ and $L^p$.   Thus, the
positive  compact   operator  $T$  is   irreducible  if  and   only  any
$T$-invariant set  $A$ is such  that either $\mu(A)=0$  or $\mu(A^c)=0$.
According      to      \cite[Theorem~3]{depagter_86}      (see      also
\cite[Section~III.3]{zerner_87}  for   an  elementary   presentation  in
$L^p$), if  $T$ is  an irreducible positive  compact operator  on $L^p$,
then either $R_0[T]>0$, or the situation is degenerate in the sense that
$\Omega$ is  an atom  of $\mu$  (that is,  for all  $A\in \cf$,  we have
either $\mu(A)=0$ or $\mu(A^c)=0$) and $T=0$.

Let $\ca$  be the set  of $T$-invariant sets,  and notice that  $\ca$ is
stable   by  countable   unions   and   countable  intersections.    Let
$\cfi=\sigma(\ca)$  be the  $\sigma$-field  generated  by $\ca$.   Then,
thanks to~\cite[Theorem~8]{schwartz61},  the operator $T$  restricted to
an atom of  $\mu$ in $\cfi$ is irreducible.  We  shall only consider non
degenerate atom, and say the atom (of  $\mu$ in $\cfi$)  is non-zero if the restriction of the
operator $T$ to this atom has a positive spectral radius.
We denote by $(\Omega_i, i\in  I)$  the at most countable (but possibly
empty) collection  of non-zero  atoms of  $\mu$ in $\cfi$. 
 Notice
that the atoms are  defined up to an a.e.\ equivalence  and can be chosen
to be pair-wise disjoint.  For $i\in I$, we set:
\begin{equation}\label{eq:ki}
  T_i=M_{\Omega_i} \, T\, M_{\Omega_i},
\end{equation}
which is a positive   compact   operator  on   $L^p$. Note that:
\begin{equation}
   \label{eq:def-T'}
  T\geq T'
  \quad\text{where}\quad
T'=  \sum_{i\in I} T_i.
\end{equation}

We now give some properties of the Frobenius decomposition. 
\begin{remark}[Properties of Frobenius decomposition]
  \label{rem:Frob0}
We have, with $i\in I$:
\begin{enumerate}[(i)]

\item By definition of the non-zero atoms:  $\mu(\Omega_i)>0$ and $T$ restricted to $\Omega_i$ is irreducible
  with   positive
  spectral radius, that is, $R_0[T_i]>0$.
\item According to \cite[Theorem~8]{schwartz61}, the spectral radius of
  $T_i$ is a simple eigenvalue of $T_i$: $\mult(R_0(T_i), T_i)=1$. 
\item According to \cite[Theorem~7]{schwartz61}, for all $\lambda\in
  \C^*$, we have:
 \begin{equation}
   \label{eq:mult-vpFrob}
   \mult(\lambda, T)= \sum_{j\in I} \mult(\lambda, T_j).
  \end{equation} 
\item    Consider    the    complement    of    the    non-zero
  atoms, say 
  $\Omega_0=\left(\cup_{j\in I} \Omega_j\right)^c$  (with the convention
  that $0$ does  not belong to the set of  indices $I$). Then, the restriction
  of    $T$    to    $\Omega_0$     is    quasi-nilpotent,    that    is
  $R_e[T](\ind{\Omega_0})=0$.  
\end{enumerate}

From those properties, we deduce the following elementary results. 

  \begin{enumerate}[(i),resume]
 \item\label{item:Frob} The cardinal of the set of indices $i\in I$ such that
    $R_0[T_i]=R_0[T]$ is exactly equal to the multiplicity of $R_0[\kk]$ for $T$,
    that is $\mult(R_0[T], T)$.
  \item   There   exists   at   least  one  non-zero   atom   ($\sharp I\geq 1$) if and only if $R_0[T]>0$.
  \item The operator $T$ is quasi-nilpotent  if and only if there is 
    no non-zero  atom ($\sharp I= 0$).

  \item \label{rem:Aci-nv} If $A\in \cf$ invariant implies $A^c$
    invariant (which is in particular the case if $T$  is self-adjoint and $p=2$), then we have
    $T=\sum_{i\in I} T_i$ and thus the restriction of $T$  to $\Omega_0$
    is zero (intuitively $T$ is block diagonal). 

   \end{enumerate}
\end{remark}

\begin{remark}\label{rem:tilde-T}
  Assume that  $T=T_\kk$ is  an integral operator  with kernel  $\kk$ on
  $\Omega=[0, 1]$;  see Remark~\ref{rem:op-noyau}.  Then,  the operators
  $T_i$ are integral  operators with respective kernel  $\kk_i$ given by
  $\kk_i(x,y)=\ind{\Omega_i}(x)  \,  \kk(x,y)\, \ind{\Omega_i}(y)$;  and
  the operator $T'=T_{  \kk'}$ is also an integral  operator with kernel
  $\kk'=\sum_{i\in I} \kk_i$.   We represent in Figure~\ref{fig:atomic1}
  an example  of a kernel  $\kk$ with  its atomic decomposition  using a
  ``nice'' order on $\Omega$ (see \cite{victory82, victory93, dowson} on
  the existence  of such  an order relation;  intuitively the  kernel is
  upper block triangular: the population on the ``left'' of an atom does
  not  infect the  population  on the  ``right'' of  this  atom) and  in
  Figure~\ref{fig:atomic2}  the corresponding  kernel  $ \kk'$.   Notice
  that $\kk(\Omega_i, \Omega_j)=0$ for $j$ ``smaller'' that $ i $, where
  $\kk(A, B)=\int_{\Omega^2} \ind{A}(x)  \kk(x,y)\ind{B}(y)\, \mu(\rd x)
  \mu(\rd y)$ is a consistent notation with~\eqref{eq:def-kT}.
\end{remark}

\begin{figure}
  \begin{subfigure}[T]{.5\textwidth}
    \centering
    \includegraphics[page=1]{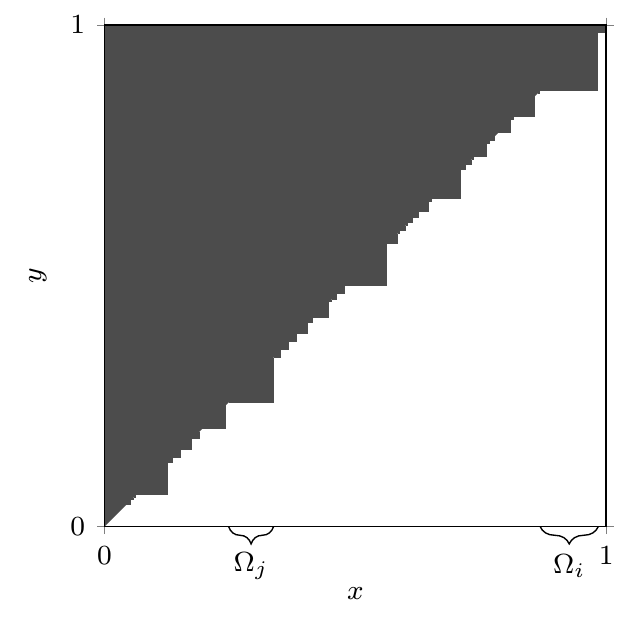}
    \caption{A representation of the kernel $\kk$ with the white zone included in
    $\{\kk=0\}$.}
    \label{fig:atomic1}
  \end{subfigure}%
  \begin{subfigure}[T]{.5\textwidth}
    \centering
    \includegraphics[page=2]{reducible}
    \caption{A representation of the kernel $\kk'=\sum_{i\in I}
      \kk_i$ with the white zone included in
    $\{\kk'=0\}$.}
    \label{fig:atomic2}
  \end{subfigure}
  \caption{Example of a  kernel $\kk$ on $\Omega=[0, 1]$  and the kernel
    $\kk'          =\sum_{i\in          I}         \kk_i$,          with
    $\kk_i(x,y)=\ind{\Omega_i}(x)\,  \kk(x,y)\,  \ind{\Omega_i}(y)$ and
    $(\Omega_i, i\in I)$ the non-zero atoms.  We
    have         $\spec[T_\kk]=\spec[T_{\kk'}]$         and         thus
    $R_e[T_\kk]=R_e[ T_{\kk'}]$.}
  \label{fig:atomic}
\end{figure}

For $i\in I$ and $\eta\in \Delta$, we set $\eta_i=\eta \ind{\Omega_i}$ and recall that
$T_i=M_{\Omega_i}\, T\, M_{\Omega_i}$. We now give the decomposition of $R_e[T]$
according to the irreducible components $(\Omega_i, i\in I)$ of $T$.

\begin{lemma}\label{lem:Rei}
Let   $T$  be   a  positive   compact   operator  on   $L^p$  for   some
$p\in  [1, +\infty  )$. With the convention that $\max_\emptyset=0$, we have
  for $\eta\in \Delta$:
  \begin{equation}\label{eq:R=maxRi}
    R_e[T](\eta)
    =\max_{i\in I} R_e[T_i](\eta_i)
    =\max_{i\in I} R_e[T_i](\eta)
    =\max_{i\in I} R_e[T](\eta \ind{\Omega_i}),
  \end{equation}
  and more generally:
  \begin{equation}
    \label{eq:mult-vpFrob3}
     \mult(\lambda, T M_\eta)=\sum_{i\in I} \mult(\lambda, T_i M_\eta)
    \quad\text{for all $\lambda\in \C^*$.}
 \end{equation} 
 \end{lemma}

Before proving the lemma, we first state a  direct consequence
 of~\eqref{eq:mult-vpFrob3},        in        the       spirit        of
 Section~\ref{sec:equivalent} on a spectrum-preserving
 transformation. Recall $T'= \sum_{i\in I} T_i$ in~\eqref{eq:def-T'}. 
\begin{corollary}
  \label{cor:Speci}
  Let   $T$  be   a  positive   compact   operator  on   $L^p$  for   some
  $p\in  [1, +\infty  )$.
  We have:
  \[
       \spec[T]= \spec[T']=\bigcup _{i\in I}
       \spec[T_i]  \quad\text{and}\quad 
      R_e[T]=R_e\left[T'\right]= \max_{i\in I}   R_e[T_i] .
    \]
\end{corollary}

\begin{proof}[Proof of Lemma~\ref{lem:Rei}]
  Let $T'$  be a positive compact  operator on $L^p$. Recall  the kernel
  $\kk_{T'}  $  defined  in~\eqref{eq:def-kT}.    For  $A\in  \cf$,  let
  $\mult(\lambda, T', A)$ denote the  multiplicity (possibly equal to 0)
  of the  eigenvalue $\lambda\in \C^*$  for the operator $T'\,  M_A$.  A
  direct  application  of  \cite[Lemma~11]{schwartz61} (which holds also
  if $\mu$ is a $\sigma$-finite measure) gives  that  for
  $A,   B\in  \cf$    such   that  $A\cap   B=\emptyset$  a.e.\   and
  $\kk_{T'}(B,A)=0$, we have for all $\lambda\in \C^*$ that:
  \begin{equation}
    \label{eq:mult=sum-mult}
    \mult (\lambda, T', A\cup B)= \mult(\lambda, T', A) + \mult(\lambda, T', B), 
  \end{equation}
  and thus
  \begin{equation}
    \label{eq:R0=maxR0}
    R_e[T'](\ind{A}+ \ind{B})=\max \big( R_e[T'](\ind{A}),R_e[T'](\ind{B})
    \big).
  \end{equation}

\medskip

  Let $A, B\in \cf$ be such that $A\cap
  B=\emptyset$ a.e.\ and $\kk_T(B,A)=0$. Let $\eta \in \Delta$. Clearly
  we have $\kk_{T M_\eta}
  (B, A)\leq\kk_T
  (B, A) $ and thus $\kk_{T M_\eta}
  (B, A) =0$. Use~\eqref{eq:mult=sum-mult} to get  that for  $\eta\in
  \Delta$ and $\lambda\in \C^*$:
  \[
    \mult(\lambda, TM_\eta, A\cup B)= \mult(\lambda, T M_ \eta, A) + \mult(\lambda,
    TM_ \eta, B). 
  \]
  Then, an immediate adaptation of the proof of \cite[Theorem~7]{schwartz61} gives that
  for all $\lambda\in \C^*$:
  \begin{equation}
    \label{eq:mult-vpFrob2}
    \mult(\lambda, TM_ \eta, \Omega)=\sum_{i\in I} \mult(\lambda, TM_ \eta, \Omega_i).
  \end{equation} 
  By definition of $\mult(\lambda, \cdot, \cdot)$, we get $
  R_e[T](\eta)=\max \{|\lambda|\, \colon\, \mult(\lambda, TM_ \eta,
  \Omega)>0\}$ and $
  R_e[TM_{\Omega_i}](\eta)= \max \{|\lambda|\, \colon\, \mult(\lambda,
  TM_ \eta, \Omega_i)>0\}$.
  This gives that:
  \[
    R_e[T](\eta)= \max_{i\in I} R_e[TM_{\Omega_i}](\eta ).
  \]
  To conclude, notice, using  Lemma~\ref{lem:hk/h}~\ref{lem:hk=kh}
  for the second equality, that:
  \[
R_e[T](\eta \ind{\Omega_i} )=R_e[TM_{\Omega_i}](\eta
  )=R_e[M_{\Omega_i} TM_{\Omega_i}](\eta )=R_e[T_i](\eta   )=R_e[T_i](\eta_i ).
\]
Similarly we deduce~\eqref{eq:mult-vpFrob3}
from~\eqref{eq:mult-vpFrob2}. 
\end{proof}

\subsection{Monatomic operators and applications}
\label{sec:q-irr-monat}

Following \cite[Definition~2.11]{bjr},  a positive compact  operator $T$ is  
\emph{quasi-irreducible} if there exists a measurable set $\oa\subset \Omega$ 
such that $\mu(\oa)>0$, $T=M_{\oa} T  M_{\oa}$ and $T$ restricted to $\oa$ is 
irreducible with positive spectral radius. The quasi-irreducible property is natural in
the setting  of positive compact  self-adjoint operators; in a more general setting,
one  would still  want to consider positive compact  operator  with  only  one
irreducible component. This motivates the next definition. Recall the atomic
decomposition of the previous section. 

\begin{definition}[Monatomic operator]\label{def:monatomic}
  Let  $T$ be a positive compact operator on $L^p$ with some $p\in [1,  +\infty )$. The
  operator is monatomic if there exists a unique non-zero atom ($\sharp I=1$).
\end{definition}

In   a   sense,  the   operator   $T$   is  ``truly   reducible''   when
$\sharp  I\geq  2$.    We  shall  give  in  a   forthcoming  work  other
characterizations of monatomic operator.
  
\begin{remark}[Link between (quasi-)irreducible and monatomic operators]\label{rem:quasi-irr}
  Irreducible   positive compact operators with    positive   spectral    radius    
   and quasi-irreducible   positive compact operators are    monatomic,
   and we have $T=\Ta$ where  $\Ta=M_{\oa} \, T\, M_{\oa}$ and $\oa$ is
   the non-zero atom, with $\oa=\Omega$ in the reducible case. 
\end{remark}

\begin{figure}
  \begin{subfigure}[T]{.5\textwidth}
    \centering
    \includegraphics[page=1]{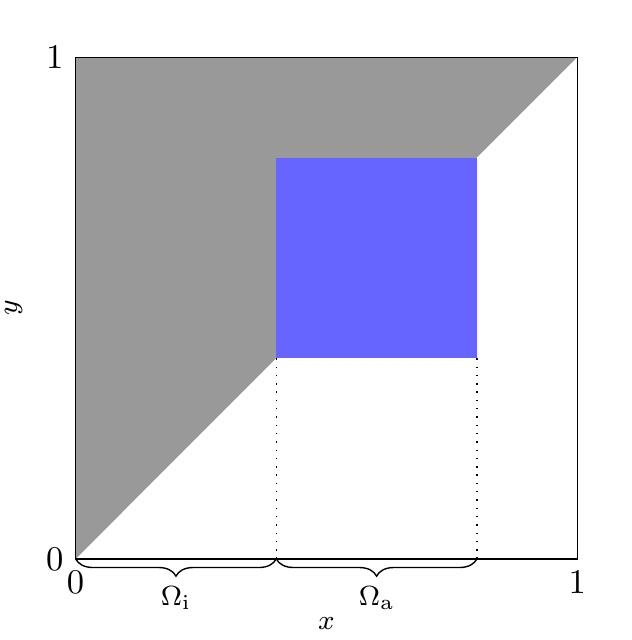}
    \caption{A representation of a monatomic kernel.}
    \label{fig:monatomic}
  \end{subfigure}%
  \begin{subfigure}[T]{.5\textwidth}
    \centering
    \includegraphics[page=2]{monatomic}
    \caption{A representation of a quasi-irreducible kernel.}
    \label{fig:quasi}
  \end{subfigure}
  \caption{Example of kernels $\kk$ and $\ka$ of a monatomic integral
    operator $T_\kk$ and the  quasi-irreducible
    integral operator $\Ta=T_{\ka}$ on $\Omega=[0, 1]$, with non-zero atom
    $\oa$. The kernels are zero on
    the white zone and  are irreducible when restricted to the zone.}
  \label{fig:monatomic-quasi}
\end{figure}

\begin{remark}[Reducibility for integral operators]\label{rem:quasi-irr2}
  We  consider  an integral  operator  $T_\kk$  with kernel  $\kk$,  see
  Remark~\ref{rem:op-noyau}, and we say the kernel $\kk$ is irreducible, quasi-irreducible
  or monatomic  whenever  the   integral operator $T_\kk$  satisfies  the corresponding
  property.   Then, the  notion of irreducibility  of a  kernel  depends only  on  its
  support.   Indeed, provided that the  measure $\mu$ is finite and the  kernel so that
  all the  operators are  well  defined  and compact,  the  kernel $\kk$  is irreducible
  (resp.\  quasi-irreducible, resp.\  monatomic) if  and only if    the   kernel   
  $\ind{\{\kk>0\}}$    is   irreducible    (resp.\ quasi-irreducible, resp.\ monatomic).  
  Furthermore, the corresponding integral operators have the same atoms.

  We  have represented  in Figure~\ref{fig:monatomic}  a monatomic  kernel $\kk$ on
  $\Omega=[0, 1]$ and  in Figure~\ref{fig:quasi} the kernel $\ka$ (with 
  $\ka(x,y)=\ind{\oa}(x)\kk(x,y)\ind{\oa}(y)$)  associated  to  the quasi-irreducible 
  integral operator  $\Ta=M_{\oa} \,  T_\kk\, M_{\oa}$; the  set   $\Omega=[0,  1]$  being
  ``nicely  ordered''  so   that  the representation of the kernels are upper triangular.
  Using the  epidemic interpretation of Remark~\ref{rem:lien-epidemie}  below, we also
  represented the subset  $\oi$ of the population infected  by the non-zero atom $\oa$. 
\end{remark}

\begin{remark}[Epidemiological interpretation]\label{rem:lien-epidemie}
  In the infinite dimensional SIS    model   developed    in
  \cite{delmas_infinite-dimensional_2020},   the  space
   $(\Omega, \cf, \mu)$ represents all the traits of the population with
   $\mu(\rd  y)$ the  infinitesimal size  of the  population with  trait
   $y$. The next-generation operator is given by  the integral operator
   $T_{\kk}$,    see     Equation~\eqref{eq:def-op-int},    where    the
   kernel~$\kk=k/\gamma$  is defined  in  terms of  a transmission  rate
   kernel   $k$   and  a   recovery   rate   function~$\gamma$  by   the
   formula~$\kk(x,y)=k(x,y)/\gamma(y)$ and has a finite double norm in
   $L^p$ for some $p\in(1, +\infty )$; and
   the basic reproduction number~$R_0=R_0[T_\kk]$ is  then the spectral radius
   of~$T_{\kk}$.  
  Intuitively,  $\kk(x,y)>0$  (resp.\  $=0$)  means  that individuals  with   trait  $y$  
  can  (resp.\ cannot)  infect individual with trait $x$.

   When the integral operator $T_\kk$ is monatomic, with  non-zero  atom $\oa$, then the population
   with trait in $\oa$ can infect  itself as well as the population with
   other distinct traits,  say $\oi$. The population with trait $\oi$ can only
   infect itself (but not $\oa$!); and there is no persistent epidemic
   outside $\oa\cup\oi$. We shall see in a forthcoming paper that the
   set $\oa\cup\oi$ is characterized as the smallest invariant set
   containing the atom $\oa$. 
\end{remark}

From Lemma~\ref{lem:Rei}, we deduce the following two results related to
monatomic operators.

\begin{lemma}\label{lem:R0simple}
  Let  $T$   be  a  positive   compact  operator  on  $L^p$   with  some
  $p\in [1,  +\infty )$, and  set $R_0=R_0[T]$.  If the operator  $T$ is
  monatomic   then   $R_0>0$   and    $R_0$   is   simple   (\textit{i.e.},
  $\mult(R_0, T)=1$).  If $R_0$ is simple and the only eigenvalue in
  $(0, +\infty )$, then the operator $T$ is monatomic.
\end{lemma}

\begin{proof}
  Let $T$ be monatomic, so that  there exists only one non-zero atom, say
  $\oa$.  Set $\Ta=M_{\oa}  T M_{\oa}$.  Since the  restriction of $\Ta$
  (or  $T$)  to  $\oa$  is  irreducible and  non-zero,  we  deduce  from
  \cite[Theorem~3]{depagter_86}  that its  spectral radius  is positive,
  and thus  $R_0[\Ta]>0$.  Using Lemma~\ref{lem:Rei}, this  implies that
  $R_0[T]=R_0[\Ta]>0$.   According  to \cite[Theorem~8]{schwartz61},  we
  get   that  $R_0[\Ta]$   is   simple  for   $\Ta$.   Since   according
  to~\eqref{eq:mult-vpFrob}  $ \mult(\lambda,  T)= \mult(\lambda,  \Ta)$
  for all $\lambda\in \C^*$, we deduce that $R_0[T]$ is simple for $T$.

  For the  second part, if  $T$ is not  monatomic and $R_0[T]>0$, we deduce  that there
  exists at  least two  non-zero atoms,  and thus  $\sharp I\geq  2$ (if
  there  is no  non-zero atom,  then  $T$ would  be quasi-nilpotent  and
  $R_0[T]=0$).  The  restrictions of  $T$ to  those non-zero  atoms have
  positive  spectral radius  according to  \cite[Theorem~3]{depagter_86}
  and  thus  at  least  one  positive  eigenvalue  by  the  Krein-Rutman
  theorem. We  deduce from~\eqref{eq:mult-vpFrob} that $T$  has at least
  two  positive eigenvalues  (counting  their multiplicity  if they  are
  equal).  This gives the result by contraposition.
\end{proof}

\begin{lemma}\label{lem:conc-mono}
  Let  $T$   be  a   positive  compact  operator   on  $L^p$   for  some
  $p\in [1, +\infty )$ such that $R_0[T]>0$.  If
  the function $R_e[T]$ is concave  on $\Delta$, then the operator $T$ 
  is monatomic.
\end{lemma}

\begin{proof}
  Since   $R_0[T]$   is   positive,   we  deduce   that   $T$   is   not
  quasi-nilpotent. Suppose  that $T$ is  not monatomic. This  means that
  the   cardinal   of  the   at   most   countable   set  $I$   in   the
  decomposition~\eqref{eq:R=maxRi} is  at least $2$. So  let $T_{1}$ and
  $T_{2}$ be  two quasi-irreducible components  of $T$, where  we assume
  that $\{1,2 \}  \subset I$.  Let $\Omega_{1}$  and $\Omega_{2}$ denote
  their respective  non-zero atoms. Without  loss of generality,  we can
  suppose that  $R_0[T_{2}] \geq R_0[T_{1}]>0$. Consider  the strategies
  $\eta_1=\ind{\Omega_{1}}$                                          and
  $\eta_2=R_0[T_{1}] \, R_0[T_{2}]^{-1}\,  \ind{\Omega_{2}}$ (which both
  belong  to  $\Delta$).   For  $\theta\in  [0,  1]$,   we  deduce  from
  \eqref{eq:R=maxRi}  and the  homogeneity of  the spectral  radius that
  $R_e[T](\theta  \eta_1  +  (1-\theta) \eta_2)=  R_e[T_1]  \max(\theta,
  1-\theta)$.  Since  $\theta\mapsto   \max(\theta,  1-\theta)$  is  not
  concave, we deduce that $R_e[T]$ is not concave on $\Delta$.
\end{proof}

\printbibliography
\end{document}

%% file: packages.tex
\usepackage[utf8]{inputenc}
\usepackage[T1]{fontenc}
\usepackage[text={460pt,665pt}, centering, a4paper, marginparwidth=2cm]{geometry}
\usepackage[hidelinks,pdfusetitle]{hyperref}
\usepackage[activate={true,nocompatibility},final,tracking=true,kerning=true,factor=1100,stretch=10,shrink=10]{microtype}
\usepackage{braket}    
\usepackage{amsthm}    
\usepackage{amssymb}   
\usepackage{mathrsfs}  
\usepackage[shortlabels]{enumitem}
\usepackage{caption}
\usepackage[margin=5pt,justification=centering,labelformat=simple,labelfont=sc]{subcaption}
\usepackage{graphicx}

\usepackage{bbold}

\theoremstyle{plain}
\newtheorem{theorem}{Theorem}[section]
\newtheorem{corollary}[theorem]{Corollary}
\newtheorem{proposition}[theorem]{Proposition}
\newtheorem{lemma}[theorem]{Lemma}

\newtheorem{definition}[theorem]{Definition}

\theoremstyle{remark}
\newtheorem{remark}[theorem]{Remark}
\newtheorem{example}[theorem]{Example}

\DeclareMathOperator\spec{Spec}               
\DeclareMathOperator\mult{m}               

\newcommand{\cb}{\ensuremath{\mathscr{B}}}

\newcommand{\cf}{\ensuremath{\mathscr{F}}}

\newcommand{\ck}{\ensuremath{\mathscr{K}}}
\newcommand{\ind}[1]{\mathbb{1}_{#1}}
\newcommand{\un}{\mathbb{1}}
\newcommand{\zero}{\mathbb{0}}
\newcommand{\ca}{\ensuremath{\mathscr{A}}}

%% file: macros.tex
\newcommand{\abs}[1]{\left\lvert\,#1\,\right\rvert}
\newcommand{\norm}[1]{\left\lVert\,#1\,\right\rVert}

\newcommand{\R}{\ensuremath{\mathbb{R}}}

\newcommand{\C}{\ensuremath{\mathbb{C}}}
\newcommand{\N}{\ensuremath{\mathbb{N}}}

\newcommand{\Diag}{\ensuremath{\mathrm{Diag}}}

\newcommand{\kk}{\ensuremath{\mathrm{k}}}
\newcommand{\rd}{\ensuremath{\mathrm{d}}}
\newcommand{\cll}{\ensuremath{\mathcal{L}}}

\newcommand{\as}{\text{ a.s.}}
\newcommand{\lb}{[\![}

\newcommand{\oa}{\ensuremath{\Omega_\mathrm{a}}}
\newcommand{\oi}{\ensuremath{\Omega_\mathrm{i}}}

\newcommand{\pos}{\ensuremath{\mathrm{p}}}
\newcommand{\nega}{\ensuremath{\mathrm{n}}}

%% file: info.tex
\date{\today}

\author{Jean-François Delmas}
\address{Jean-François Delmas,
  CERMICS, \'{E}cole des Ponts, France}
\email{jean-francois.delmas@enpc.fr}

\author{Dylan Dronnier}
\address{Dylan Dronnier,
  Université de Neuchâtel, Switzerland}
\email{dylan.dronnier@unine.ch}

\author{Pierre-André Zitt}
\address{Pierre-André Zitt, LAMA, Université Gustave Eiffel, France}
\email{pierre-andre.zitt@univ-eiffel.fr}

%% file: local.tex
\newcommand{\ci}{\mathcal{I}}
\newcommand{\Ta}{\ensuremath{T_\mathrm{a}}}
\newcommand{\ka}{\ensuremath{\kk_\mathrm{a}}}
\newcommand{\cfi}{\ensuremath{\cf_\mathrm{inv}}}